\documentclass[11pt,a4paper,reqno]{amsart} 
\usepackage[applemac]{inputenc}
\usepackage{amsmath}
\usepackage{amssymb}
\usepackage{euscript}
\usepackage{mathrsfs}
\usepackage{wasysym}
\usepackage[all]{xy}
\usepackage{graphicx}
\usepackage{color}
\usepackage{multirow}
\usepackage{pifont}
\usepackage[table]{xcolor}
\usepackage[colorlinks=true,linktocpage=true,pagebackref=false, urlcolor=black, citecolor=black,linkcolor=black]{hyperref}
\usepackage{scalerel}[2014/03/10]
\usepackage{stackengine}

\newcommand\widesim[1]{\ThisStyle{%
 \setbox0=\hbox{$\SavedStyle#1$}%
 \stackengine{-.1\LMpt}{$\SavedStyle#1$}{%
 \stretchto{\scaleto{\SavedStyle\mkern.2mu\sim}{.4\wd0}}{.4\ht0}%
 }{O}{c}{F}{T}{S}%
}}

\newtheorem{thm}{Theorem}[section]

\newtheorem{lem}[thm]{Lemma}
\newtheorem{prop}[thm]{Proposition}

\newtheorem{cor}[thm]{Corollary}

\theoremstyle{definition}

\theoremstyle{remark}
\newtheorem{remark}[thm]{Remark}
\newtheorem{remarks}[thm]{Remarks}
\newtheorem{example}[thm]{Example}

\numberwithin{equation}{section}
\renewcommand{\theparagraph}{\thesubsection.\arabic{paragraph}}
\newcommand{\K}{{\mathbb K}} \newcommand{\N}{{\mathbb N}}
 \newcommand{\R}{{\mathbb R}}
 \newcommand{\C}{{\mathbb C}}

\newcommand{\gtp}{{\mathfrak p}} \newcommand{\gtq}{{\mathfrak q}}
\newcommand{\gtm}{{\mathfrak m}} \newcommand{\gtn}{{\mathfrak n}}
\newcommand{\gta}{{\mathfrak a}} \newcommand{\gtb}{{\mathfrak b}}
\newcommand{\gtP}{{\mathfrak P}} \newcommand{\gtQ}{{\mathfrak Q}}

\newcommand{\Fhaz}{{\EuScript F}}
\newcommand{\Jhaz}{{\EuScript J}}
\newcommand{\Pp}{{\EuScript P}}
\newcommand{\Ss}{{\EuScript S}}
\newcommand{\Tt}{{\EuScript T}}
\newcommand{\Uu}{{\EuScript U}}

\newcommand{\Hw}{{\EuScript H}_w}
\newcommand{\an}{{\EuScript O}}
\newcommand{\mer}{{\EuScript M}}
\newcommand{\I}{{\EuScript I}}
\newcommand{\ceros}{{\EuScript Z}}

\newcommand{\Reg}{\operatorname{Reg}}
\newcommand{\Sing}{\operatorname{Sing}}
\newcommand{\qf}{\operatorname{qf}}
\newcommand{\Int}{\operatorname{Int}}
\newcommand{\cl}{\operatorname{Cl}}

\newcommand{\supp}{\operatorname{supp}}
\newcommand{\Spec}{\operatorname{Spec}}

\newcommand{\id}{\operatorname{id}}

\newcommand{\x}{{\tt x}} \newcommand{\y}{{\tt y}} 
\newcommand{\z}{{\tt z}} \renewcommand{\t}{{\tt t}}
 
\newcommand{\w}{{\tt w}}

\newcommand{\ol}{\overline}

\numberwithin{equation}{section}

\renewcommand\thesubsection{\thesection.\Alph{subsection}}

\begin{document}
\title[Normalization of complex analytic spaces from a global viewpoint]{Normalization of complex analytic spaces\\ from a global viewpoint}

\author{Francesca Acquistapace}
\author{Fabrizio Broglia}
\address{Dipartimento di Matematica, Universit\`a degli Studi di Pisa, Largo Bruno Pontecorvo, 5, 56127 PISA (ITALY)}
\email{francesca.acquistapace@unipi.it, fabrizio.broglia@unipi.it}

\author{Jos\'e F. Fernando}
\address{Departamento de \'Algebra, Facultad de Ciencias Matem\'aticas, Universidad Complutense de Madrid, 28040 MADRID (SPAIN)}
\email{josefer@mat.ucm.es}

\date{10/10/2017}

\thanks{Authors supported by Spanish GRAYAS MTM2014-55565-P. First and second authors are also supported by ``National Group for Algebraic and Geometric Structures, and their Applications'' (GNSAGA - INdAM) and Italian MIUR. Third author is also supported by Grupos UCM 910444. This article is the fruit of the close collaboration of the authors in the last fifteen years and has been mainly written during a two-months research stay of third author in the Dipartimento di Matematica of the Universit\`a di Pisa. Third author would like to thank the department for the invitation and the very pleasant working conditions.}

\begin{abstract}--- \ 
In this work we study some algebraic and topological properties of the ring $\an(X^\nu)$ of global analytic functions of the normalization $(X^\nu,\an_{X^\nu})$ of a reduced complex analytic space $(X,\an_X)$. If $(X,\an_X)$ is a Stein space, we characterize $\an(X^\nu)$ in terms of the (topological) completion of the integral closure $\ol{\an(X)}^\nu$ of the ring $\an(X)$ of global holomorphic functions on $X$ (inside its total ring of fractions) with respect to the usual Fr\'echet topology of $\ol{\an(X)}^\nu$. This shows that not only the Stein space $(X,\an_X)$ but also its normalization is completely determined by the ring $\an(X)$ of global analytic functions on $X$. This result was already proved in 1988 by Hayes-Pourcin when $(X,\an_X)$ is an irreducible Stein space whereas in this paper we afford the general case. We also analyze the real underlying structures $(X^\R,\an_X^\R)$ and $(X^{\nu\,\R},\an_{X^\nu}^\R)$ of a reduced complex analytic space $(X,\an_X)$ and its normalization $(X^\nu,\an_{X^\nu})$. We prove that the complexification of $(X^{\nu\,\R},\an_{X^\nu}^\R)$ provides the normalization of the complexification of $(X^\R,\an_X^\R)$ if and only if $(X^\R,\an_X^\R)$ is a coherent real analytic space. Roughly speaking, coherence of the real underlying structure is equivalent to the equality of the following two combined operations: (1) normalization + real underlying structure + complexification, and (2) real underlying structure + complexification + normalization.
\end{abstract}
\maketitle

\section{Introduction}\label{s1}

In this paper we analyze the algebraic and topological relation between the Fr\'echet algebras $\an(X)$ and $\an(X^\nu)$ of global analytic functions on a reduced complex analytic space $(X,\an_X)$ and on its normalization $((X^\nu,\an_{X^\nu}),\pi)$. Algebraic operations as integral closure and topological operations as completion will have significant roles. Given a commutative ring $A$ we denote $\ol{A}^\nu$ the integral closure of $A$ in its total ring of fractions $Q(A)$. 

Assume first that $X\subset\C^n$ carries an algebraic structure and let $(X^\mu\subset\C^{n+m},\rho)$ denote its algebraic normalization. The ring $\Pp(X^\mu)$ of polynomial functions on $X^\mu$ is (isomorphic to) the integral closure $\ol{\Pp(X)}^\nu$ of the ring $\Pp(X)$ of polynomial functions on $X$ (in the ring of rational functions on $X$). The ring $\ol{\Pp(X)}^\nu$ is in addition a finitely generated $\Pp(X)$-module and a reduced $\C$-algebra. We endow $X$ and $X^\mu$ with their natural analytic structures, that is, we consider the reduced complex analytic spaces $(X,\an_X:=\an_{\C^n}|_X)$ and $(X^\mu,\an_{X^\mu}:=\an_{\C^{n+m}}|_{X^\mu})$. Zariski Main Theorem is equivalent to the following statement: \em $\an_{X^\mu,y}$ is an integrally closed domain for each $y\in X^\mu$ \em (see \cite[\S V.6]{mo}). As a consequence of the latter fact, one can prove that $\an(X^\mu)$ is (isomorphic to) the integral closure $\overline{\an(X)}^\nu$ of $\an(X)$ in its total ring of fractions. More precisely, one has the following result.

\begin{thm}[Zariski Theorem]\label{algebraiccase}
Let $X\subset\C^n$ be an algebraic set. Then the $\C$-algebras $\ol{\an(X)}^\nu$ and $\an(X^\mu)$ are isomorphic, the tuple $(X^\mu,\an_{X^\mu},\rho)$ is isomorphic to the (analytic) normalization $((X^\nu,\an_{X^\nu}),\pi)$ of $(X,\an_X)$ and $\an(X^\nu)$ is a finitely generated $\an(X)$-module.
\end{thm}
 
This fact is no longer true in general when $(X,\an_X)$ is a reduced complex analytic space, even if $(X,\an_X)$ is a $1$-dimensional Stein space, as it is shown in \cite[\S1]{hp} exhibiting an explicit counterexample. Alternatively, in Example \ref{dim2} we provide a $2$-dimensional Stein space $(X,\an_X)$ such that $\ol{\an(X)}^\nu\neq\an(X^\nu)$.

In this paper we prove that if $(X,\an_X)$ is a reduced Stein space, the $\C$-algebra $\an(X^\nu)$ is the completion as a metric space of the integral closure $\overline{\an(X)}^\nu$ of $\an(X)$ in its total ring of fractions. The latter coincides with the ring $\mer(X)$ of meromorphic functions on $X$ because $(X,\an_X)$ is a Stein space \cite[52.17, 53.1]{kk}. The space $\an(X^\nu)$ is endowed with the natural metric topology of uniform convergence on compact sets \cite[8.3]{gare}. As the sheaf $\an_{X^\nu}$ is coherent and $X^\nu$ is separable, $\an(X^\nu)$ with such topology is a Fr\'echet space \cite[VIII.A.Thm.8]{gr2}. The inclusion $\ol{\an(X)}^\nu\hookrightarrow\an(X^\nu)$ (see Theorem \ref{normalization0}) endows $\ol{\an(X)}^\nu$ with the induced (metric) topology. If $S\subset\an(X^\nu)$, the closure $\cl(S)$ coincides with the completion of $S$ as a metric space, so we can recover $\cl(S)$ from $S$ without referring to the ambient space $\an(X^\nu)$. In general, $\ol{\an(X)}^\nu$ needs not to be complete \cite[\S1]{hp} and if such is the case $\ol{\an(X)}^\nu\neq\an(X^\nu)$. However, as announced above, we show in Section \ref{s3} the following result.

\begin{thm}[Density of the integral closure]\label{main}
Let $(X,\an_X)$ be a reduced Stein space. Then $\an(X^\nu)$ is the completion of the (metric) space $\ol{\an(X)}^\nu$, or equivalently, $\ol{\an(X)}^\nu$ is a dense subset of $\an(X^\nu)$. 
\end{thm}

The result above was proved in \cite {hp} only for irreducible Stein spaces. Our proof in the general case is quite different and is obtained as a consequence of a user-friendly description (as an inverse limit) of the closure of an $\an(X)$-submodule ${\mathfrak N}$ of the $\an(X)$-module of global sections $H^0(X,\Fhaz)$ of a coherent sheaf of $\an_X$-modules $\Fhaz$ (the $\an(X)$-module $H^0(X,\Fhaz)$ is endowed with its natural Fr\'echet topology \cite[VIII.A.Thm.8]{gr2}). More precisely, let $K\subset X$ be a compact set and let $\Ss_K$ be the set of holomorphic functions on $X$ whose zero-sets do not meet $K$. The set $\Ss_K$ is a multiplicatively closed set, which may contain zero divisors if $X$ is not irreducible. We call $\Ss_K^{-1}{\mathfrak N}$ the \em module of fractions of ${\mathfrak N}$ associated with $K$\em. If $K_1\subset K_2\subset X$ are compact sets, then $\Ss_{K_2}\subset\Ss_{K_1}$ and 
$$
\rho_{K_1,K_2}:\Ss_{K_2}^{-1}{\mathfrak N}\to\Ss^{-1}_{K_1}{\mathfrak N},\ \frac{F}{H}\mapsto\frac{F}{H}
$$ 
is a homomorphism of $\an(X)$-modules, which may be non-injective if $X$ is not irreducible. Obviously if $K,K'\subset X$ are compact sets, their union $K'':=K\cup K'$ is a compact subset of $X$ that contains both. Thus, the multiplicatively closed sets $\Ss_K$ allow us to represent $\cl({\mathfrak N})$ as the following inverse limit.

\begin{prop}\label{inverse}
The closure $\cl({\mathfrak N})$ is (isomorphic to) the inverse limit $
{\displaystyle\lim_{\substack{\longleftarrow\\K\subset X\\\text{compact}}}}\Ss_K^{-1}{\mathfrak N}$ of the inverse system 
$$
{\mathfrak S}:=\Big\{\{\Ss_K^{-1}{\mathfrak N}\}_{\substack{K\subset X\\\text{compact}}},\{\rho_{K_1,K_2}\}_{\substack{K_1\subset K_2\subset X\\\text{compact}}}\Big\}.
$$
In addition, if $\{K_\ell\}_{\ell\geq1}$ is an exhaustion of $X$ by compact sets, then $\cl({\mathfrak N})\cong\displaystyle\lim_{\substack{\longleftarrow\\\ell\geq1}}\Ss_{K_\ell}^{-1}{\mathfrak N}$.
\end{prop}

As a straightforward application of Proposition \ref{inverse} (making ${\mathfrak N}=\an(X)$), we write the ring $\an(X)$ as the inverse limit of the rings of fractions $\Ss_K^{-1}\an(X)$ where each $K\subset X$ is a compact set. 

\begin{cor}\label{inverse-c0}
The ring $\an(X)$ is (isomorphic to) the inverse limit ${\displaystyle\lim_{\substack{\longleftarrow\\K\subset X\\\text{compact}}}}\Ss_K^{-1}\an(X)$ of the directed system 
$$
{\mathfrak S}:=\Big\{\{\Ss_K^{-1}\an(X)\}_{\substack{K\subset X\\\text{compact}}},\{\rho_{K_1,K_2}\}_{\substack{K_1\subset K_2\subset X\\\text{compact}}}\Big\}.
$$
In addition, if $\{K_\ell\}_{\ell\geq1}$ is an exhaustion of $X$ by compact sets, then $\an(X)\cong\displaystyle\lim_{\substack{\longleftarrow\\\ell\geq1}}\Ss_{K_\ell}^{-1}\an(X)$.
\end{cor}
We call $\Ss_K^{-1}\an(X)$ the \em ring of fractions of $\an(X)$ associated to $K$\em. Corollary \ref{inverse-c0} generalizes the fact that
\begin{equation}\label{anxeq1}
\an(X)=\bigcap_{\substack{K\subset X\\\text{compact}}}\Ss_K^{-1}\an(X)
\end{equation}
if $X$ is irreducible. To prove \eqref{anxeq1} from Corollary \ref{inverse-c0} one uses that each ring of fractions $\Ss_K^{-1}\an(X)$ is a subring of $\mer(X)$, that each homomorphism $\rho_{K_1,K_2}:\Ss_{K_2}^{-1}\an(X)\to\Ss^{-1}_{K_1}\an(X)$ is injective if $K_1\subset K_2\subset X$ and the following well-known remark.
\begin{remark}\label{limint}
Let $\{A_i\}_{i\in I}$ be a family of subgroups of a group $A$ such that for each pair $i,j\in I$ there exists $k\in I$ such that $A_k\subset A_i\cap A_j$. We consider the partial order $\leq$ on $I$ given by: \em $i\leq k$ if and only if $A_k\subset A_i$\em, and the family of inclusion homomorphisms ${\tt j}_{ik}:A_k\hookrightarrow A_i$ if $i\leq k$. Observe that ${\tt j}_{ii}=\id_{A_i}$ and ${\tt j}_{ik}={\tt j}_{i\ell}\circ{\tt j}_{\ell k}:A_k\hookrightarrow A_\ell\hookrightarrow A_i$ if $i\leq \ell\leq k$. Thus, the pair $(\{A_i\}_{i\in I},\{{\tt j}_{ik}\}_{i\leq k})$ is an inverse system of groups and inclusion homomorphisms. The inverse limit $\displaystyle\lim_{\substack{\longleftarrow\\i\in I}}A_i$ is the intersection $A:=\bigcap_{i\in I}A_i$ together with the inclusion homomorphisms ${\tt j}_i:A\hookrightarrow A_i$ for each $i\in I$.\qed
\end{remark}

If $X$ is not irreducible, the homomorphisms $\rho_{K_1,K_2}:\Ss_{K_2}^{-1}\an(X)\to\Ss^{-1}_{K_1}\an(X)$ may not be injective and the rings of fractions $\Ss_K^{-1}\an(X)$ may not be subrings of $\mer(X)$, because the multiplicatively closed set $\Ss_K$ may contain zero divisors, so the intersection works no longer and inverse limit is required for the description of $\an(X)$. 

The customary description of Stein spaces as being those complex spaces that have ``sufficiently many'' global holomorphic functions \cite[VII]{gr2} attains precision from a theorem of Forster/Igusa/Iwahashi/Remmert \cite{of,ig,i,r} stating: \em $(X,\an_X)$ is Stein if and only if the map $\chi:X\to\Spec_c(\an(X)),\ x\mapsto\chi_x$, that maps a point $x\in X$ to the evaluation homomorphism $\chi_x:\an(X)\to\C,\ f\mapsto f(x)$ at such point, is a homeomorphism\em. Recall that $\Spec_c(\an(X))$ is the set of continuous $\C$-algebra homomorphisms $\varphi:\an(X)\to\C$. Consequently, a complex analytic space $(X,\an_X)$ is Stein if and only if there exist enough global holomorphic functions on $X$ to enable $X$ to be regained topologically from the continuous spectrum of these functions, that is, $(X,\an_X)$ is completely determined by its $\C$-algebra $\an(X)$ of global holomorphic functions. In addition, a reduced complex analytic space is Stein if and only if its normalization is Stein \cite{n3}. Thus, we conclude that if $(X,\an_X)$ is a reduced Stein space, its $\C$-algebra $\an(X)$ of global analytic functions {\em determines} both spaces $(X,\an_X)$ and $(X^\nu,\an_{X^\nu})$.

We include in \ref{ssas} for the sake of completeness a proof of Zariski Theorem \ref{algebraiccase} that follows the ideas developed in this work. Such result compares the algebraic structure of an algebraic set and its underlying complex analytic structure. In a similar way we may compare the real underlying structures of a complex analytic space and its normalization. A celebrated approach to complex algebraic curves (which can be understood as particular cases of Stein spaces) is the study of their real underlying structures from which the concept of \em Riemannian surface \em arises. There are certain properties that are preserved when considering the real underlying structure of a complex analytic space: local regularity, local irreducibility, local normality, etc. 

We analyze the behavior of the normalization of a reduced complex analytic space $(X,\an_X)$ when considering the real underlying structure $(X^\R,\an_X^\R)$. By \ref{uras} and \ref{complexification} there exists a \em complexification \em of $(X^\R,\an_X^\R)$, that is, a complex analytic space $(\widesim{X^\R},\an_{\widesim{X}^\R})$ endowed with an anti-involution $\sigma:\widesim{X^\R}\to\widesim{X^\R}$ such that $X^\R$ is the set of fixed points of $\sigma$. If $(X^\nu,\an_{X^\nu},\pi)$ is the normalization of $(X,\an_X)$ and $(X^{\nu\,\R},\an_{X^\nu}^\R)$ is its real analytic structure, we complexify the real analytic morphism $\pi^\R:X^{\nu\,\R}\to X^\R$ to obtain a complex analytic morphism $\widesim{\pi^\R}:\widesim{X^{\nu\,\R}}\to{\widesim{X^\R}}$, where $(\widesim{X^{\nu\,\R}},\an_{\widesim{X^\nu}^\R})$ is a complexification of $(X^\nu,\an_{X^\nu})$, see \ref{com}. We have the following commutative diagram. 
$$
\xymatrix{
{X^\nu}\ar[d]_\pi\ar@{<->}[r]&{X^{\nu\,\R}}\ar[d]_{\pi^\R}\ar@{^{(}->}[r]&{\widesim{X^{\nu\,\R}}}\ar[d]_{\widesim{\pi^\R}}\\
X\ar@{<->}[r]&X^\R\ar@{^{(}->}[r]&{\widesim{X^\R}}
}
$$
Our next result determines when the tuple $(\widesim{X^{\nu\,\R}},\an_{\widesim{X^{\nu\,\R}}},\widesim{\pi^\R})$ is the normalization of $({\widesim{X^\R}},\an_{\widesim{X^\R}})$ in terms of the coherence of the real analytic space $(X^\R,\an_{X^{\R}})$.

\begin{thm}[Real underlying structure of the normalization]\label{neighnorm0}
The following assertions are equivalent:
\begin{itemize}
\item[(i)] The real analytic space $(X^\R,\an_{X^{\R}})$ is coherent.
\item[(ii)] For each point $a\in X$ the irreducible components of the germ $X_a$ remain irreducible in a neighborhood of $a$. 
\item[(iii)] The triple $(\widesim{(X^{\nu})^\R},\an_{\widesim{(X^{\nu})^\R}},\widesim{\pi^\R})$ is the normalization of $(\widesim{X^\R},\an_{\widesim{X^\R}})$ after shrinking $\widesim{(X^{\nu})^\R}$ and $\widesim{X^\R}$ if necessary.
\end{itemize}
\end{thm}
Roughly speaking the previous result determines under what extent the operations of normalization and complexification commute via considering the real underlying structure (in the proper place), that is,
$$
\left\{\!\!
\begin{array}{l}
\hspace*{3mm}\text{ real structure}\\ 
+\text{ complexification}\\ 
+\text{ normalization} 
\end{array}\!\!\right\}
= 
\left\{\!\!
\begin{array}{l}
\hspace*{3mm}\text{ normalization}\\ 
+\text{ real structure}\\ 
+\text{ complexification}
\end{array}\!\!\right\}
\iff 
\left\{\!\!
\begin{array}{c}
\text{real structure}\\ 
\text{is coherent}
\end{array}\!\!\right\}
$$
or in other words
$$
\Big(\widesim{X^\R}\Big)^\nu =\widesim{(X^\nu)^\R}\iff X^\R\text{ is coherent}.
$$

\subsection*{Structure of the article}
The article is organized as follows. In Section \ref{s2} we present all basic terminology and notations used in this article as well as some basic results concerning holomorphic and anti-holomorphic functions on complex analytic spaces. The reading can be started directly in Section \ref{s3} and referred to the Section \ref{s2} of `Basic facts' only when needed. In Section \ref{s3} we prove Theorems \ref{algebraiccase} and \ref{main} and Proposition \ref{inverse}. To that end, given a Stein space $(X,\an_X)$ we represent the ring $\an(X)$ as an inverse limit of excellent rings (Corollary \ref{inverse-c0} and Theorem \ref{excell0}) and the ring $\an(X^\nu)$ as an inverse limit of normal excellent rings (Corollary \ref{inverse-c0} and Theorem \ref{norexcell}). In addition, we provide an explicit $2$-dimensional Stein space $(X,\an_X)$ such that $\ol{\an(X)}^\nu\neq\an(X^\nu)$ (Example \ref{dim2}). In Section \ref{s5} we analyze the real underlying structure of the normalization of a complex analytic space in order to prove Theorem \ref{neighnorm0}. The proof of this result requires a deep knowledge of some local properties of the real underlying structure of a complex analytic set explored in Section \ref{s4}.

\section{Basic facts on real and complex analytic spaces}\label{s2}

We collect next notations and basic facts that are recurrent in the article.

\subsection{Notations and general terminology}\label{below1}

In the following \em holomorphic \em refers to the complex case whereas \em analytic \em to the real case. For a further reading about complex analytic spaces we refer to \cite{gr2} while we remit the reader to \cite{gmt} for the theory of real analytic spaces. We denote the elements of $\an(X)$ with capital letters if $(X,\an_X)$ is a complex analytic space and with small letters if $(X,\an_X)$ is a real analytic space. We will use freely Remmert's Theorem \cite[VII.\S2.Thm.2]{n} that states: \em The image of a proper holomorphic map between complex analytic spaces is a complex analytic set\em. 

Denote the coordinates in $\C^n$ with $z:=(z_1,\ldots,z_n)$ where $z_i:=x_i+\sqrt{-1}y_i$. Consider the conjugation $\ol{\,\cdot\,}:\C^n\to\C^n,\ z\mapsto\ol{z}:=(\ol{z_1},\ldots,\ol{z_n})$ of $\C^n$, whose set of fixed points is $\R^n$. A subset $S\subset\C^n$ is \em invariant \em if $\ol{S}=S$. Let $\Omega\subset\C^n$ be an invariant open set and $F:\Omega\to\C$ a holomorphic function. We say that $F$ is \em invariant \em if $F(z)=\ol{F(\ol{z})}$ for each $z\in\Omega$. This implies that $F$ restricts to a real analytic function on $\Omega\cap\R^n$. Conversely, if $f$ is analytic on $\R^n$, it can be extended to an invariant holomorphic function $F$ on some invariant open neighborhood $\Omega\subset\C^n$ of $\R^n$. 

\subsection{Real and imaginary parts}
Write the tuple $z:=(z_1,\ldots,z_n)\in\C^n$ as $z=x+\sqrt{-1}y$ where $x:=(x_1,\ldots,x_n)$ and $y:=(y_1,\ldots,y_n)$, so we identify $\C^n$ with $\R^{2n}$. If $F:\Omega\to\C$ is a holomorphic function, $F(x+\sqrt{-1}y):=\Re^*(F)(x,y)+\sqrt{-1}\Im^*(F)(x,y)$ where 
$$
\Re^*(F)(x,y):=\frac{F(z)+\ol{F(z)}}{2}\quad\text{and}\quad\Im^*(F)(x,y):=\frac{F(z)-\ol{F(z)}}{2\sqrt{-1}}
$$
are real analytic functions on $\Omega\equiv\Omega^\R$ understood as an open subset of $\R^{2n}$. Assume in addition that $\Omega$ is invariant. Then
$$
\Re(F):\Omega\to\C,\ z\mapsto\frac{F(z)+\ol{F(\ol{z})}}{2}\quad\text{and}\quad\Im(F):\Omega\to\C,\ z\mapsto\frac{F(z)-\ol{F(\ol{z})}}{2\sqrt{-1}}
$$ 
are invariant holomorphic functions that satisfy $F=\Re(F)+\sqrt{-1}\,\Im(F)$. We have
$$
\Re^*(F)=\Re^*(\Re(F))-\Im^*(\Im(F))\quad\text{and}\quad \Im^*(F)=\Im^*(\Re(F))+\Re^*(\Im(F)),
$$
so it is convenient not to confuse the pair of real analytic functions $(\Re^*(F),\Im^*(F))$ on $\Omega^\R$ with the pair of invariant holomorphic functions $(\Re(F),\Im(F))$ on $\Omega$. 

\subsection{Reduced analytic spaces \em \cite[I.1]{gmt}} 
Let $\K=\R$ or $\C$ and let $(X,\an_X)$ be an either complex or real analytic space. Let $\Fhaz_X$ be the sheaf of $\K$-valued functions on $X$ and let $\vartheta:\an_X\to\Fhaz_X$ be the morphism of sheaves defined for each open set $U\subset X$ by $\vartheta_U(s):U\to\K,\ x\mapsto s(x)$ where $s(x)$ is the class of $s$ module the maximal ideal $\gtm_{X,x}$ of $\an_{X,x}$. Recall that $(X,\an_X)$ is \em reduced \em if $\vartheta$ is injective. Denote the image of $\an_X$ under $\vartheta$ with $\an_X^r$. The pair $(X,\an_X^r)$ is called the \em reduction \em of $(X,\an_X)$ and $(X,\an_X)$ is reduced if and only if $\an_X=\an_X^r$. The reduction is a covariant functor from the category of $\K$-analytic spaces to that of reduced $\K$-analytic spaces.

\subsection{Normalization of reduced complex analytic spaces}

A reduced complex analytic space $(X,\an_X)$ is \em normal \em if for each point $x\in X$, the local analytic ring $\an_{X,x}$ is a \em normal ring\em, that is, it is reduced and integrally closed (in its total ring of fractions). Riemmann's extension theorem holds for normal complex analytic spaces, that is, if $X$ is a normal complex analytic space and $Y\subset X$ is a closed analytic subset of codimension greater than or equal to one, each function $F$ on $X$ that is holomorphic outside $Y$ and locally bounded at the points of $Y$ can be extended holomorphically to the whole $X$. Functions on an arbitrary complex analytic space that are holomorphic outside some closed analytic subset of codimension at least one and are locally bounded at the points of such closed analytic subset are usually called \em weakly holomorphic\em. We can restate Riemann's extension theorem as follows: {\em Each weakly holomorphic function on a normal complex analytic space is holomorphic.} 

We consider for each reduced complex analytic space $(X,\an_X)$ the following two sheaves of $\an_X$-modules.
\begin{itemize}
\item The sheaf $\Hw$ of germs of weakly holomorphic functions.
\item The {\em normalization sheaf} $\an^\nu_X$, whose fibre at each point $x\in X$ is the integral closure $\ol{\an_{X,x}}^\nu$ of $\an_{X,x}$ in its total ring of fractions $Q(\an_{X,x})$.
\end{itemize}
Recall that if $\gtp_1,\ldots,\gtp_s$ are the minimal prime ideals of the ring $\an_{X,x}$, the total ring of fractions $Q(\an_{X,x})$ of $\an_{X,x}$ is (isomorphic to) the product of the fields of fractions $Q(\an_{X,x}/\gtp_j)$ of the integral domains $\an_{X,x}/\gtp_j$. In addition, the integral closure of the local ring $\an_{X,x}$ in its total ring of fractions is (isomorphic to) the product of the integral closures of the rings $\an_{X,x}/\gtp_i$ in their respective fields of fractions.

According to \cite{gare,ok2} both sheaves of $\an_X$-modules $\Hw$ and $\an^\nu_X$ are coherent. We summarize the main results concerning normalization of reduced complex analytic spaces in the following theorem.

\begin{thm}[Normalization]\label{normalization0} 
Let $(X,\an_X)$ be a reduced complex analytic space. Then there exists a normal complex analytic space $(X^\nu,\an_{X^\nu})$ together with a proper (surjective) holomorphic map $\pi:X^\nu\to X$ that is a $1$-sheeted analytic ramified cover, whose critical set is the set of singular points $\Sing(X)$ of $X$. The couple $((X^\nu,\an_{X^\nu}),\pi)$ is unique up to biholomorphic diffeomorphism. In addition,
\begin{itemize}
\item[(i)] If $X_x=X_{1,x}\cup\cdots\cup X_{s,x}$ is the decomposition into irreducible components of the germ $X_x$ of $X$ at a point $x\in X$, the fiber $\pi^{-1}(x)$ has cardinal $s$ and (after reordering the indices) $\pi$ maps a neighborhood of the point $y_j\in\pi^{-1}(x)$ in $X^\nu$ onto a neighborhood of $x$ in a representative of $X_{j,x}$ for $j=1,\ldots,s$. Additionally, if $\gtp_1,\ldots,\gtp_s$ are the minimal prime ideals of the ring $\an_{X,x}$, then (after reordering the indices) $\an_{X^\nu,y_j}\cong\ol{\an_{X,x}/\gtp_j}^\nu$ for $j=1,\ldots,s$.
\item[(ii)] The homomorphism of rings $\pi^*:\mer(X)\to\mer(X^\nu),\ \xi\mapsto\xi\circ\pi$ induced by $\pi$ is an isomorphism between the rings of meromorphic functions of $X$ and $X^\nu$. We say that $(X,\an_X)$ and $(X^\nu,\an_{X^\nu})$ are \em birational\em.
\item[(iii)] The coherent sheaves of $\an_X$-modules $\Hw$, $\an^\nu_X$ and $\pi_*(\an_{X^\nu})$ are isomorphic. 
\item[(iv)] We have the following chain of inclusions and isomorphisms
$$
\an(X)\hookrightarrow\ol{\an(X)}^\nu\hookrightarrow H^0(X,\Hw)\cong H^0(X,\an^\nu_{X})\cong H^0(X,\pi_*(\an_{X^\nu}))=\an(X^\nu).
$$
\end{itemize}
\end{thm}

For the proof of the previous theorem we refer the reader to \cite[\S6 \& \S8]{gare} or \cite[Lemme fondamental]{ok2}. The space $(X^\nu,\an_{X^\nu})$ is called the {\em normalization of $(X,\an_X)$}. We recall next how to show elementarily that: \em An integral element $\xi\in Q(\an(X))$ over $\an(X)$ is a weakly holomorphic function\em, that is, $\ol{\an(X)}^\nu\subset H^0(X,\Hw)$.

\begin{proof}
Let $F,G,A_0,\ldots,A_{p-1}\in\an(X)$ be holomorphic functions such that $G$ is a non zero divisor of $\an(X)$ and $\xi=\frac{F}{G}$ satisfies the monic polynomial equation $\t^p+\sum_{k=0}^{p-1}A_k\t^k$. Observe that $\xi$ is holomorphic on $X\setminus\{G=0\}$ and, as $G$ is a non zero divisor of $\an(X)$, the analytic set $Y:=\{G=0\}$ has codimension $\geq1$. Let us show that $\xi$ is locally bounded at the points of $Y$. Pick a point $x\in X\setminus Y$, where $\xi$ is defined and $\xi(x)\neq0$. We have
\begin{multline*}
\xi(x)^p+\sum_{k=0}^{p-1}A_k(x)\xi(x)^k\ \leadsto\ \xi(x)^p=-\sum_{k=0}^{p-1}A_k(x)\xi(x)^k\\ \leadsto\ |\xi(x)|\leq\sum_{k=0}^{p-1}|A_k(x)||\xi(x)|^{k-p-1}\ \leadsto\ |\xi(x)|\leq\max\Big\{1,\sum_{k=0}^{p-1}|A_k(x)|\Big\}.
\end{multline*}
As the function $\max\{1,\sum_{k=0}^{p-1}|A_k|\}:X\to\R$ is continuous, $\xi$ is locally bounded at the points of $Y$, as required.
\end{proof}

Using Cartan's Theorem B and \cite[52.17, 53.1]{kk} one shows that if $(X,\an_X)$ is a reduced Stein space, the total ring of fractions $Q(\an(X))$ coincides with the ring $\mer(X)$ of meromorphic functions on $X$. Thus, we have the following diagram
$$
\xymatrix{
\an(X)\ar@{^{(}->}[r]&\ol{\an(X)}^\nu\ar@{^{(}->}[r]\ar@{^{(}-->}[d]&\an({X^\nu})\ar@{^{(}->}[d]\\
&Q(\an(X))=\mer(X)\ar@{<->}[r]^\cong_{\pi^*}&\mer({X^\nu})=Q(\an({X^\nu}))
}
$$
that is, $\an(X)\hookrightarrow\ol{\an(X)}^\nu\hookrightarrow\an(X^\nu)\hookrightarrow\mer(X)$.

\subsection{Underlying real analytic structure \em \cite[II.4]{gmt}}\label{uras} 
Let $(Z,\an_Z)$ be a local model for a complex analytic space defined by a coherent sheaf of ideals $\I\subset\an_{\C^n}|_{\Omega}$, that is, $Z:=\supp(\an_{\C^n}|_{\Omega}/\I)$ and $\an_Z:=(\an_{\C^n}|_{\Omega}/\I)|_Z$. Suppose that $\I$ is generated by finitely many holomorphic functions $F_1,\ldots,F_r$ on $\Omega$. Let $\I^\R$ be the coherent sheaf of ideals of $\an_{\R^{2n}}|_{\Omega^\R}$ generated by $\Re^*(F_i),\Im^*(F_i)$ for $i=1,\ldots,r$. Let $(Z^\R,\an_Z^\R)$ be the local model for a real analytic space defined by the coherent sheaf of ideals $\I^\R$. For each complex analytic space $(X,\an_X)$ there exists a structure of real analytic space on $X$ that we denote $(X^\R,\an_X^\R)$ and it is called the \em real underlying structure of $(X,\an_X)$\em. The previous construction provides a covariant functor from the category of complex analytic spaces to that of real analytic spaces \cite[I.3.3]{gmt}. If $(X,\an_X)$ is a reduced complex analytic space, it may fail that $(X^\R,\an^\R_X)$ is coherent or reduced \cite[III.2.15]{gmt}. 

\subsection{Holomorphic and anti-holomorphic sections} The conjugation in $\C$ induces readily a \em conjugation \em in the sheaf $\an_X^\R\otimes_\R\C$ as follows. First, if $(Z,\an_Z)$ is a local model for complex analytic spaces, we define the conjugate germ of $F_x:=\Re^*(F_x)+\sqrt{-1}\Im^*(F_x)\in\an_{Z,x}^\R\otimes_\R\C$ as $\ol{F_x}:=\Re^*(F_x)-\sqrt{-1}\Im^*(F_x)\in\an_{Z,x}^\R\otimes_\R\C$. We define the conjugation in the sheaf $\an_X^\R\otimes_\R\C$ by considering neighborhoods at each point $x\in X$ that are isomorphic to some local model. The conjugation $\omega:\an_X^\R\otimes_\R\C\to\an_X^\R\otimes_\R\C$ is the morphism of sheaves such that $\omega_x(F_x)=\ol{F_x}$ for each $x\in X$ and $F_x\in\an_{X,x}^\R\otimes_\R\C$.

A germ $G_x\in\an_{X,x}^\R\otimes_\R\C$ is called \em anti-holomorphic \em (resp. \em holomorphic\em) if there exists an isomorphism of a neighborhood of $x$ onto a local model such that $G_x$ is the image of an anti-holomorphic (resp. holomorphic) germ. Of course $G_x$ is anti-holomorphic if and only if $\ol{G_x}$ is holomorphic. We denote by $\ol{\an}_X$ the sheaf of anti-holomorphic sections that we define as a subsheaf of $\an_X^\R\otimes_\R\C$. For each $U\subset X$ open we have
$$
H^0(U,\ol{\an}_X):=\{G\in H^0(U,\an_X^\R\otimes_\R\C):\ G_x\text{ is anti-holomorphic }\forall x\in U\}.
$$
The sheaf $\an_X$ of holomorphic sections may be regarded analogously as a subsheaf of $\an_X^\R\otimes_\R\C$. The conjugation of $\an_X^\R\otimes_\R\C$ turns holomorphic sections into anti-holomorphic ones and vice versa. If a germ $G_x\in\an_{X,x}\cap\ol{\an}_{X,x}$, there exists $H_x\in\an_{X,x}$ such that $G_x=\ol{H_x}$. Thus, $|G|^2_x=G_x\ol{G_x}=G_xH_x\in\an_{X,x}$, so $|G|^2_x$ is by Remmert's Theorem constant and by the maximum modulus principle $G_x\in\C$. This means that if $U\subset X$ is an open subset, $H^0(U,\an_X)\cap H^0(U,\ol{\an}_X)=\C$.

\subsubsection{Anti-involutions}\label{complex}
Let $(X,\an_X)$ be a complex analytic space and let $(X^\R,\an_X^\R)$ be its real underlying structure. An anti-involution on $(X,\an_X)$ is a morphism of $\R$-ringed spaces $\sigma:(X^\R,\an_X^\R\otimes_\R\C)\to(X^\R,\an_X^\R\otimes_\R\C)$ such that $\sigma^2=\id$ and it interchanges the subsheaf of holomorphic sections $\an_X$ with the of subsheaf anti-holomorphic ones $\ol{\an}_X$. For simplicity the anti-involutions shall be denoted as $\sigma:(X,\an_X)\to(X,\an_X)$.

\subsubsection{Fixed part space}\label{fixed}
Let $(X,\an_X)$ be a complex analytic space endowed with an anti-involu\-tion $\sigma:(X,\an_X)\to(X,\an_X)$. Let $X^\sigma:=\{x\in X:\ \sigma(x)=x\}$ and define a sheaf $\an_{X^\sigma}$ on $X^\sigma$ in the following way: for each open subset $U\subset X^\sigma$, we define $H^0(U,\an_{X^\sigma})$ as the subset of $H^0(U,\an_X|_{X^\sigma})$ of invariant sections. The $\R$-ringed space $(X^\sigma,\an_{X^\sigma})$ is called the fixed part space of $(X,\an_X)$ with respect to $\sigma$. By \cite[II.4.10]{gmt} it holds that $(X^\sigma,\an_{X^\sigma})$ is a real analytic space if $X^\sigma\neq\varnothing$ and it is a closed subspace of $(X^\R,\an_X^\R)$.

\subsection{Complexification and $C$-analytic spaces \cite[III.3]{gmt}}\label{complexification} 
A real analytic space $(X,\an_X)$ is a \em $C$-analytic space \em if it satisfies one of the following two equivalent conditions:
\begin{itemize}
\item[(1)] Each local model of $(X,\an_X)$ is defined by a coherent sheaf of ideals.
\item[(2)] There exist a complex analytic space $(\widesim{X},\an_{\widesim{X}})$ endowed with an anti-involution $\sigma$ whose fixed part space is $(X,\an_X)$.
\end{itemize}
We call the complex analytic space $(\widesim{X},\an_{\widesim{X}})$ is a \em complexification \em of $X$. As $(X,\an_X)$ is a coherent real analytic space (because the local models are defined by coherent sheaves of ideals), the complexfication $(\widesim{X},\an_{\widesim{X}})$ satisfies the following properties:
\begin{itemize}
\item[(i)] $\an_{\widesim{X},x}=\an_{X,x}\otimes_\R\C$ for each $x\in X$.
\item[(ii)] The germ of $(\widesim{X},\an_{\widesim{X}})$ at $X$ is unique up to an isomorphism.
\item[(iii)] $X$ has a fundamental system of invariant open Stein neighborhoods in $\widesim{X}$. 
\item[(iv)] If $X$ is reduced, then $\widesim{X}$ is also reduced.
\end{itemize} 
For further details see \cite{c2,gmt,t}.
\begin{remark}
The above definition of complexification differs from Cartan's classical one. In \cite{c2} a \em complexification \em of a reduced real analytic space is a reduced complex analytic space $(\widesim{X},\an_{\widesim{X}})$ satisfying conditions (i), (ii) and (iii) above.
\end{remark} 

\subsubsection{Complexification of morphisms \em \cite[III.3.11]{gmt}}\label{com}
Let $\varphi:(X,\an_X)\to(Y,\an_Y)$ be a morphism of $C$-analytic spaces. Let $(\widesim{X},\an_{\widesim{X}})$ and $(\widesim{Y},\an_{\widesim{Y}})$ be respective complexifications of $(X,\an_X)$ and $(Y,\an_Y)$. There exist:
\begin{itemize}
\item[(i)] a Stein open neighborhood $\Omega\subset\widesim{X}$ of $X$ and an anti-involution $\sigma:(\Omega,\an_{\widesim{X}}|_\Omega)\to(\Omega,\an_{\widesim{X}}|_\Omega)$ whose fixed part space is $(X,\an_X)$,
\item[(ii)] a Stein open neighborhood $\Theta\subset\widesim{Y}$ of $Y$ and an anti-involution $\tau:(\Theta,\an_{\widesim{Y}}|_\Theta)\to(\Theta,\an_{\widesim{Y}}|_\Theta)$ whose fixed part space is $(Y,\an_Y)$,
\item[(iii)] a morphism of Stein spaces $\widesim{\varphi}:(\Omega,\an_{\widesim{X}}|_\Omega)\to(\Theta,\an_{\widesim{Y}}|_\Theta)$ such that $\widesim{\varphi}|_X=\varphi$ and $\widesim{\varphi}^\R\circ\sigma=\tau\circ\widesim{\varphi}^\R$.
\end{itemize}
In addition, if $\varphi$ is an isomorphism (resp. embedding), shrinking $\Omega$ and $\Theta$, also $\widesim{\varphi}$ is an isomorphism (resp. embedding). We show next that if $\widesim{\varphi}^{-1}(Y)=X$ and $\varphi$ is proper and surjective, shrinking $\Omega$ and $\Theta$, also $\widesim{\varphi}$ is proper and surjective.

\begin{lem}\label{neighcomplex}
Let $(X,\an_X)$ and $(Y,\an_Y)$ be $C$-analytic spaces. Let $\varphi:X\to Y$ be a proper surjective analytic map and let $\widesim{\varphi}:\widesim{X}\to\widesim{Y}$ be a complexification of $\varphi$. Suppose that $\widesim{\varphi}^{-1}(Y)=X$. Then there exist open neighborhoods $\Omega\subset\widesim{X}$ of $X$ and $\Theta\subset\widesim{Y}$ of $Y$ such that $\widesim{\varphi}:\Omega\to\Theta$ is proper and surjective.
\end{lem}
\begin{proof}
Let $\{L_k\}_{k\geq1}$ be an exhaustion of $Y$ by compact sets. As $\varphi$ is proper, $\{K_k:=\varphi^{-1}(L_k)\}_{k\geq1}$ is an exhaustion of $X$ by compact sets. As $\widesim{Y}$ is paracompact and locally compact, there exists a locally finite covering $\{V_i\}_{i\in I}$ of $\widesim{Y}$ by open subsets with compact closures. For each $x\in X$ let $W^x\subset\widesim{X}$ be and open neighborhood of $x$ with compact closure such that there exists $i(x)\in I$ satisfying $\widesim{\varphi}(\cl(W^x))\subset U^x:=V_{i(x)}$. As $\varphi$ is surjective, $Y\subset\bigcup_{x\in X}U^x$.

Consider the family of compact sets $\{K_k\setminus\Int(K_{k-1})\}_{k\geq1}$ where $K_0=\varnothing$. For each $k\geq1$ we choose a finite set $J_k$ such that 
$$
J_k\subset K_k\setminus\Int(K_{k-1})\subset\bigcup_{x\in J_k}W^x.
$$

\renewcommand{\theparagraph}{\thesubsubsection.\arabic{paragraph}}
\setcounter{secnumdepth}{4}
\paragraph{}\label{stepcito}
Define $J:=\bigcup_{k\geq1}J_k$ and let us check: \em For each $y\in\widesim{Y}$ there exists an open neighborhood $V^y\subset\widesim{Y}$ and $\ell\geq 1$ such that $V^y\cap\widesim{\varphi}(\cl(W^x))=\varnothing$ if $x\in J\setminus\bigcup_{1\leq k\leq\ell}J_k$\em. Consequently, \em the family ${\mathfrak F}:=\{\widesim{\varphi}(\cl(W^x))\}_{x\in J}$ is locally finite \em and the set \em $S:=\bigcup_{x\in J}\widesim{\varphi}(\cl(W^x))$ is closed in $\widesim{Y}$\em. 

If $y\in\widesim{Y}\setminus\cl(S)$, the result is clear. Pick a point $y\in\cl(S)$ and let $V^y\subset\widesim{Y}$ be an open neighborhood of $y$ with compact closure. As the family $\{V_i\}_{i\in I}$ is locally finite, there exists a finite set $F\subset I$ such that $V_i\cap\cl(V^y)=\varnothing$ if $i\in I\setminus F$. Consider the compact set $\bigcup_{i\in F}\cl(V_i)$ and let $\ell\geq1$ be such that $\bigcup_{i\in F}\cl(V_i)\subset\Int(L_\ell)$. Let $x\in J$ be such that $V^y\cap\widesim{\varphi}(\cl(W^x))\neq\varnothing$. Then $V^y\cap U^x\neq\varnothing$, so there exists $i\in F$ such that $U^x=V_i$. Thus, $\varphi(x)\in V_i\subset\Int(L_\ell)$, so $x\in\bigcup_{1\leq k\leq\ell}J_k$.

\paragraph{}Define $T:=\bigcup_{x\in J}\cl(W^x)$. We claim: \em the map $\widesim{\varphi}|_T:T\to S$ is proper and surjective\em. 

Let $L\subset S$ be a compact set and let us check that $(\widesim{\varphi}|_T)^{-1}(L)$ is also compact. By \ref{stepcito} there exists $\ell\geq1$ such that
$$
L\cap\bigcup_{x\in J_k, k\geq\ell+1}\widesim{\varphi}(\cl(W^x))=\varnothing.
$$
Consequently, $(\widesim{\varphi}|_T)^{-1}(L)\cap\bigcup_{x\in J_k, k\geq\ell+1}\ol{W}^y=\varnothing$. Thus, $(\widesim{\varphi}|_T)^{-1}(L)\subset\bigcup_{x\in J_k, 1\leq k\leq\ell}\cl(W^x)$ is compact because it is a closed subset of a finite union of compact sets.

\paragraph{}Let $\Omega_0:=\bigcup_{x\in J}\Int(\cl(W^x))$, which is an open neighborhood of $X$ in $\widesim{X}$, and let $C:=T\setminus\Omega_0$, which is a closed subset of $T$ that do not meet $X$. As $\widesim{\varphi}|_T$ is proper and $\widesim{\varphi}^{-1}(Y)=X$, the image $\widesim{\varphi}(C)$ is a closed subset of $\widesim{Y}$ that do not meet $Y$. Let $\Theta_1:=\widesim{Y}\setminus\widesim{\varphi}(C)$, which is an open neighborhood of $Y$ in $\widesim{Y}$. Then $\Omega_1:=(\widesim{\varphi}|_T)^{-1}(\Theta_1)\subset\Omega_0$ is an open neighborhood of $X$ in $\widesim{X}$ and the map $\varphi|_{\Omega_1}:\Omega_1\to\Theta_1$ is proper. Consequently, $\varphi(\Omega_1)$ is by Remmert's Theorem an analytic subset of $\Theta_1$ that contains $Y$. Thus, $\varphi(\Omega_1)$ is a neighborhood of $Y$ in $\widesim{Y}$. Let $\Theta:=\Int(\varphi(\Omega_1))$ and $\Omega:=(\varphi|_{\Omega_1})^{-1}(\Theta)$. The restriction $\varphi|_{\Omega}:\Omega\to\Theta$ is proper and surjective, as required.
\end{proof}
\renewcommand{\theparagraph}{\thesubsection.\arabic{paragraph}}
\setcounter{secnumdepth}{4}

\section{Normalization of Stein spaces}\label{s3}

In the first part of this section we prove Theorem \ref{main} and Proposition \ref{inverse}. We will show that the ring of fractions $\Ss_K^{-1}\an(X)$ is an excellent ring (Theorem \ref{excell0}) and the integral closure of $\Ss_K^{-1}\an(X)$ in $\mer(X)$ is $\Tt_{K^*}^{-1}\an(X^\nu)$ (Theorem \ref{norexcell}) where $\Tt_{K^*}$ is the multiplicatively closed set of all holomorphic functions on $X^\nu$ that do not vanish at $K^*:=\pi^{-1}(K)$ and $(X^\nu,\an_{X^\nu},\pi)$ denotes the normalization of $(X,\an_X)$. Recall that as $(X,\an_X)$ is a Stein space, the total ring of fractions $Q(\an(X))$ coincides with the ring of meromorphic function $\mer(X)$ on $X$. In the second part of this section we include a proof of Zariski Theorem \ref{algebraiccase} and an example of a $2$-dimensional Stein space $(X,\an_X)$ for which $\overline{\an(X)}^\nu$ is different from $\an(X^\nu)$.

\begin{proof}[Proof of Theorem \em \ref{main}]
By Theorem \ref{normalization0} (iv) the inclusion $\ol{\an(X)}^\nu\hookrightarrow\an(X^\nu)$ holds. We will show in Theorem \ref{norexcell} that 
\begin{equation}\label{norexcelleq}
\Ss_K^{-1}\ol{\an(X)}^\nu\cong\Tt_{K^*}^{-1}\an(X^\nu)
\end{equation} 
for each compact set $K\subset X$, where $K^*:=\pi^{-1}(K)$. By Proposition \ref{inverse} and (its straightforward application) Corollary \ref{inverse-c0} we conclude
\begin{equation*}
\cl(\ol{\an(X)}^\nu)\cong{\displaystyle\lim_{\substack{\longleftarrow\\K\subset X\\\text{compact}}}}\Ss_K^{-1}\ol{\an(X)}^\nu
\cong{\displaystyle\lim_{\substack{\longleftarrow\\K\subset X\\\text{compact}}}}\Tt_{K^*}^{-1}\an(X^\nu)\cong{\displaystyle\lim_{\substack{\longleftarrow\\L\subset X^\nu\\\text{compact}}}}\Tt_{L}^{-1}\an(X^\nu)\cong\an(X^\nu).
\end{equation*}
The first isomorphism in the previous row follows from Proposition \ref {inverse} applied to $\ol{\an(X)}^\nu$, the second isomorphism follows from \eqref{norexcelleq}, the third is true because the family of compact sets $K^*$ for $K\subset X$ compact is cofinal inside the family of compact subsets of $X^\nu$, whereas the last isomorphism is a consequence of Corollary \ref {inverse-c0}.
\end{proof}

Thus, in order to have Theorem \ref{main} proved we are only left to show that Proposition \ref {inverse} and Theorem \ref{norexcell} hold.

\subsection{Fr\'echet closure of an $\an(X)$-submodule}
Let $(X,\an_X)$ be a Stein space and let $\Fhaz$ be a coherent sheaf of $\an_X$-modules. Recall that we endow $H^0(X,\Fhaz)$ with its natural Fr\'echet topology \cite[VIII.A.Thm.8]{gr2}. Let ${\mathfrak N}$ be an $\an(X)$-submodule of $H^0(X,\Fhaz)$. In what follows $K$ will always denote a compact set. Consider
\begin{align*}
{\mathfrak C}_1({\mathfrak N}):=\{&A\in{H^0(X,\Fhaz)}:\ \forall K\subset X\ \exists H\in\an(X)\ \text{such that}\ \{H=0\}\cap K=\varnothing\ 
\text{and}\ HA\in {\mathfrak N}\},\\
{\mathfrak C}_2({\mathfrak N}):=\{&A\in H^0(X,\Fhaz):\ \forall x\in X\ \exists G\in\an(X)\ \text{such that}\ G(x)\neq0\ \text{and}\ GA\in {\mathfrak N}\}.
\end{align*}

\begin{lem}\label{closure}
The closure of ${\mathfrak N}$ in $H^0(X,\Fhaz)$ is $\cl({\mathfrak N})={\mathfrak C}_1({\mathfrak N})={\mathfrak C}_2({\mathfrak N})$.
\end{lem}
\begin{proof}
As the chain of inclusions ${\mathfrak C}_1({\mathfrak N})\subset{\mathfrak C}_2({\mathfrak N})\subset\cl({\mathfrak N})$ holds by \cite[VIII.Thm.4]{c1}, we only check: $\cl({\mathfrak N})\subset{\mathfrak C}_1({\mathfrak N})$. 

Let $K\subset X$ be a compact set. As $(X,\an_X)$ is a Stein space, we may assume that $K$ is holomorphically convex \cite[VII.A]{gr2}. Since ${\mathfrak N}\an_X$ is a coherent sheaf of $\an_X$-modules, there exists by Cartan's Theorem A an open neighborhood $\Omega\subset X$ of $K$ and $A_1,\ldots,A_r\in H^0(X,\Fhaz)$ such that ${\mathfrak N}\an_{X,x}$ is generated as an $\an_{X,x}$-module by $A_{1,x},\ldots,A_{r,x}$ for each $x\in\Omega$. By \cite[VIII.Thm.11]{c1} the finitely generated $\an(X)$-submodule ${\mathfrak M}:=(A_1,\ldots,A_r)\an(X)$ of $H^0(X,\Fhaz)$ is closed and by \cite[\S2.Satz 3]{of} the ideal 
$$
({\mathfrak M}:\cl({\mathfrak N})):=\{H\in\an(X):\ H\cl({\mathfrak N})\subset{\mathfrak M}\}
$$
is closed. By \cite[VIII.Lem.6]{c1} ${\mathfrak N}\an_{X,x}=\cl({\mathfrak N})\an_{X,x}$ for each $x\in X$. Thus, $({\mathfrak M}:\cl({\mathfrak N}))\an_{X,x}=\an_{X,x}$ for each $x\in\Omega$, that is, it is generated by $1$ at each point of $\Omega$. After shrinking $\Omega$, we may assume that it is a holomorphically convex neighborhood of $K$ and $H^0(\Omega,({\mathfrak M}:\cl({\mathfrak N}))\an_X)=H^0(\Omega,\an_X)$ (see \cite[VII.A.Prop.3 \& VIII.A.Thm.15]{gr2}). By \cite[VIII.A.Thm.11]{gr2} there exist holomorphic functions $H\in H^0(X,({\mathfrak M}:\cl({\mathfrak N}))\an_X)=({\mathfrak M}:\cl({\mathfrak N}))$ that are arbitrarily close to $1$ on $K$. Thus, there exists $H\in\an(X)$ such that $\{H=0\}\cap K=\varnothing$ and $H\cl({\mathfrak N})\subset{\mathfrak M}\subset{\mathfrak N}$. Consequently, $\cl({\mathfrak N})\subset{\mathfrak C}_1({\mathfrak N})$, as required. 
\end{proof}

We are ready to prove Proposition \ref{inverse}.

\begin{proof}[Proof of Proposition \em \ref{inverse}]
For each compact set $K\subset X$ consider the homomorphism 
$$
\phi_K:\cl({\mathfrak N})\to\Ss_K^{-1}{\mathfrak N}, F\mapsto\frac{FH}{H}
$$ 
where $H\in\Ss_K$ satisfies $FH\in{\mathfrak N}$ (recall that by Lemma \ref{closure} $\cl({\mathfrak N})={\mathfrak C}_2({\mathfrak N})$). It is straightforward to show that $\phi_K$ is well-defined. As $\cl({\mathfrak N})={\mathfrak C}_2({\mathfrak N})$, it holds that $\cl({\mathfrak N})$ together with the homomorphisms $\{\phi_K\}_{\substack{K\subset X\\\text{compact}}}$ is isomorphic to the inverse limit of the inverse system ${\mathfrak S}$. For the second part of the statement observe that the collection $\{K_\ell\}_\ell$ is cofinal inside the family of compact subsets of $X$.
\end{proof}
\begin{remark}\label{finitelimit}
An analogous result holds if we consider only finite subsets of $X$ (that are obviously compact sets) instead of all compact subsets of $X$. To that end use the letter ${\mathscr F}\subset X$ to denote a finite set and consider
$$
{\mathfrak C}_3({\mathfrak N}):=\{A\in{H^0(X,\Fhaz)}:\ \forall {\mathscr F}\subset X\ \exists H\in\an(X)\ \text{such that}\ \{H=0\}\cap{\mathscr F}=\varnothing\ 
\text{and}\ HA\in {\mathfrak N}\}.
$$
As ${\mathfrak C}_2({\mathfrak N})\subset{\mathfrak C}_3({\mathfrak N})\subset{\mathfrak C}_1({\mathfrak N})$ and ${\mathfrak C}_1({\mathfrak N})={\mathfrak C}_2({\mathfrak N})=\cl({\mathfrak N})$, we have $\cl({\mathfrak N})={\mathfrak C}_3({\mathfrak N})$. Once this equality holds, the alternative description 
$$
\cl({\mathfrak N})={\displaystyle\lim_{\substack{\longleftarrow\\{\mathscr F}\subset X\\\text{finite}}}}\Ss_{\mathscr F}^{-1}{\mathfrak N},
$$ 
which involves only finite sets, follows similarly to the one in Proposition \ref{inverse}.
\end{remark}

We denote an arbitrary maximal ideal of $\an(X)$ with $\gtm$ whereas $\gtm_x$ refers to the maximal ideal associated with a point $x\in X$.
\begin{cor}\label{closures}
Assume in addition that $X$ is irreducible and ${\mathfrak N}$ is a torsion-free $\an(X)$-module. We have:
\begin{itemize}
\item[(i)] ${\mathfrak N}=\bigcap_\gtm{\mathfrak N}_{\gtm}$.
\item[(ii)] $\cl({\mathfrak N})={\mathfrak C}_2({\mathfrak N})=\bigcap_{x\in X}{\mathfrak N}_{\gtm_x}$. Consequently, ${\mathfrak C}_2({\mathfrak N})={\mathfrak N}$ if and only if ${\mathfrak C}_2({\mathfrak N})\subset{\mathfrak N}_{\gtm}$ for each (free) maximal ideal $\gtm$ not associated with a point $x\in X$.
\item[(iii)] $\cl({\mathfrak N})={\mathfrak C}_1({\mathfrak N})=\bigcap_{\substack{K\subset X\\\text{compact}}}\Ss_K^{-1}{\mathfrak N}$.
\item[(iv)] If $\{K_\ell\}_{\ell\geq1}$ is an exhaustion of $X$ by compact sets, then ${\mathfrak C}_1({\mathfrak N})=\bigcap_{\ell\geq1}\Ss_{K_\ell}^{-1}{\mathfrak N}$. 
\item[(v)] If $K\subset X$ is a holomorphically convex compact set, then $\gtm\cap\Ss_K=\varnothing$ if and only if $\gtm=\gtm_x$ for some $x\in K$. In addition, $\Ss_K^{-1}{\mathfrak N}=\bigcap_{x\in K}{\mathfrak N}_{\gtm_x}$.
\end{itemize}
\end{cor}
\begin{proof}
Statements (ii) and (iv) are clear once the remaining ones are proved. Consider the multiplicatively closed set $\Ss:=\an(X)\setminus\{0\}$.

(i) Observe that ${\mathfrak N}\hookrightarrow{\mathfrak N}_{\gtm}\hookrightarrow\Ss^{-1}{\mathfrak N}$ for each maximal ideal $\gtm$ of $\an(X)$. Consequently, ${\mathfrak N}\hookrightarrow\bigcap_\gtm{\mathfrak N}_{\gtm}$. Let $\xi\in\bigcap_\gtm{\mathfrak N}_{\gtm}$ and define $\gta:=\{H\in\an(X):\ H\xi\in{\mathfrak N}\}$. We claim: $\gta=\an(X)$. 

Otherwise, there exists a maximal ideal $\gtm_0$ of $\an(X)$ such that $\gta\subset\gtm_0$. As $\xi\in{\mathfrak N}_{\gtm_0}$, there exists $A\in{\mathfrak N}$ and $H\in\an(X)\setminus\gtm_0$ such that $\xi=\frac{A}{H}$. As $\an(X)$ is an integral domain, $H\xi=A\in{\mathfrak N}$, so $H\in\gta\subset\gtm_0$, a contradiction. Consequently, $\gta=\an(X)$ and $\xi=1\cdot \xi\in{\mathfrak N}$.

(iii) As $\Ss_K\subset\Ss$, the ring $\an(X)$ is an integral domain and ${\mathfrak N}$ is torsion-free, the homomorphism $\Ss_K^{-1}{\mathfrak N}\hookrightarrow\Ss^{-1}{\mathfrak N}$ is an inclusion. In addition, $\rho_{K_1,K_2}:\Ss_{K_2}^{-1}{\mathfrak N}\to\Ss_{K_1}^{-1}{\mathfrak N}$ is injective if $K_1\subset K_2\subset X$ are compact sets. The statement follows from Proposition \ref{inverse} and Remark \ref{limint}.

(v) Let $\gtm$ be a maximal ideal of $\an(X)$. 

\paragraph{}We prove first: \em $\gtm\cap\Ss_K=\varnothing$ if and only if there exists $x\in K$ such that $\gtm=\gtm_x$\em. 

Let $\gtm$ be a maximal ideal such that $\gtm\cap\Ss_K=\varnothing$ and assume that $\gtm\neq\gtm_x$ for each $x\in K$. Thus, for each $x\in X$ there exists $F_x\in\gtm$ such that $F_x(x)\neq0$. As $K$ is compact, we find $F_1,\ldots,F_r\in\gtm$ such that $\{F_1=0,\ldots,F_r=0\}\cap K=\varnothing$. By \cite[VIII.Thm.11]{c1} the finitely generated ideal $\gta:=(F_1,\ldots,F_r)\an(X)\subset\gtm$ is closed. We have $\gta\an_{X,x}=\an_{X,x}$ for each $x\in\Omega=X\setminus\{F_1=0,\ldots,F_r=0\}$.

\paragraph{}\label{conto}We claim: \em There exists $H\in\gta$ that is close to $1$ on $K$, so $\{H=0\}\cap K=\varnothing$\em.

After shrinking $\Omega$, we may assume that it is a holomorphically convex neighborhood of $K$ and $H^0(\Omega,\gta\an_X)=H^0(\Omega,\an_X)$ (see \cite[VII.A.Prop.3 \& VIII.A.Thm.15]{gr2}). By \cite[VIII.A.Thm.11]{gr2} there exists a holomorphic function $H\in H^0(X,\gta\an_X)=\gta$ that is close to $1$ on $K$, so $\{H=0\}\cap K=\varnothing$. 

Thus, $H\in\gtm\cap\Ss_K$, which is a contradiction. Consequently, $\gtm=\gtm_x$ for some $x\in K$. The converse is clear.

\paragraph{} We check now: $\Ss_K^{-1}{\mathfrak N}=\bigcap_{x\in K}{\mathfrak N}_{\gtm_x}$. By (i) the torsion-free $\an(X)$-module $\Ss_K^{-1}{\mathfrak N}$ satisfies
$$
\Ss_K^{-1}{\mathfrak N}=\bigcap_\gtm(\Ss_K^{-1}{\mathfrak N})_\gtm\subset\bigcap_{x\in K}(\Ss_K^{-1}{\mathfrak N})_{\gtm_x}=\bigcap_{x\in K}{\mathfrak N}_{\gtm_x}\subset\mer(X).
$$
Pick now a fraction $\frac{A}{F}\in\bigcap_{x\in K}{\mathfrak N}_{\gtm_x}$. For each $x\in K$ there exists $A_x\in{\mathfrak N}$ and $F_x\not\in\gtm_x$ such that $\frac{A_x}{F_x}=\frac{A}{F}$. As $K$ is compact, we find $x_1,\ldots,x_r\in K$ such that $\{F_{x_1}=0,\ldots,F_{x_r}=0\}\cap K=\varnothing$. Proceeding as in the proof of \ref{conto} there exist $G_1,\ldots,G_r\in\an(X)$ such that the zero-set of $H:=F_{x_1}G_1+\cdots+F_{x_r}G_r$ does not meet $K$. Define $B:=A_{x_1}G_1+\cdots+A_{x_r}G_r\in{\mathfrak N}$ and observe that $\frac{A}{F}=\frac{B}{H}\in\Ss_K^{-1}{\mathfrak N}$ (because $\frac{A_{x_i}}{F_{x_i}}=\frac{A}{F}$) for each $i$. Thus, $\bigcap_{x\in K}(\Ss_K^{-1}{\mathfrak N})_{\gtm_x}\subset\Ss_K^{-1}{\mathfrak N}$, as required.
\end{proof}

\begin{remarks}
(i) By Remark \ref{finitelimit} we have $\an(X)={\displaystyle\lim_{\substack{\longleftarrow\\{\mathscr F}\subset X\\\text{finite}}}}\Ss_{\mathscr F}^{-1}\an(X)$.

(ii) If $X$ is irreducible, we have by Corollary \ref{closures}
$$
\an(X)=\bigcap_{\substack{K\subset X\\\text{compact}}}\Ss_K^{-1}\an(X)=\bigcap_{\substack{{\mathscr F}\subset X\\\text{finite}}}\Ss_{\mathscr F}^{-1}\an(X)=\bigcap_{x\in X}\an(X)_{\gtm_x}
$$ 

(iii) If $X$ has finitely many irreducible components and a compact set $K\subset X$ meets all of them, $\Ss_K$ does not contain zero divisors and $\Ss_K^{-1}\an(X)\subset\mer(X)$. In this case, we consider an exhaustion $\{K_\ell\}_{\ell\geq1}$ of $X$ by compact sets such that $K_1$ meets all the irreducible components of $X$. Then the homomorphisms $\rho_{K_\ell,K_j}:\Ss_{K_j}^{-1}\an(X)\to\Ss_{K_\ell}^{-1}\an(X)$ are injective and $\an(X)=\bigcap_{\ell\geq1}\Ss_{K_\ell}^{-1}\an(X)$.

(iv) If $X$ has infinitely many connected components, $\Ss_K$ meets the set of (non-trivial) zero divisors of $\an(X)$ because $K$ does not meet all the irreducible components of $X$. 
\end{remarks}

\subsection{Excellence of rings of fractions associated with compact sets} 
Let $(X,\an_X)$ be a Stein space and let $K\subset X$ be a compact set. 

\begin{thm}\label{excell0}
The ring of fractions $\Ss_K^{-1}\an(X)$ is excellent.
\end{thm}

To lighten the proof of the previous result we do before some preliminary work. Given a subset $S\subset X$, we denote $\Jhaz(S)$ the ideal of all holomorphic sections in $\an(X)$ that vanishes identically on $S$. As $X$ is a Stein space, $\Jhaz(S)$ generates by Cartan's Theorem A the coherent sheaf of ideals $\Jhaz_{S,x}:=\{f_x\in\an_{X,x}:\ S_x\subset\{f_x=0\}\}$.

\begin{lem}\label{elimsob}
Assume $K$ is holomorphically convex and let $Z$ be the union of the (finitely many) irreducible components of $X$ that meet $K$. Denote 
\begin{align*}
&\Ss_K:=\{F\in\an(X):\ \{F=0\}\cap K=\varnothing\},\\
&\Ss_K':=\{F\in\an(Z):\ \{F=0\}\cap K=\varnothing\}. 
\end{align*}
Then $\Ss_K^{-1}\mer(X)\cong\mer(Z)$ and $\Ss_K^{-1}\an(X)\cong{\Ss_K'}\!^{-1}\an(Z)$. In addition, $\Ss_K^{-1}\mer(X)$ is the total ring of fractions of $\Ss_K^{-1}\an(X)$.
\end{lem}
\begin{proof}
Let $\{X_k\}_{k\geq1}$ be the collection of the irreducible components of $X$. We may assume that $Z=\bigcup_{k=1}^mX_k$. Let $\Omega$ be a holomorphically convex neighborhood of $K$ in $X$ that do not meet $Z':=\bigcup_{k\geq m+1}X_k$ (see \cite[VII.A.Prop.3]{gr2}). We have $\Jhaz(Z')\an_{X,x}=\an_{X,x}$ for each $x\in\Omega$, that is, it is generated by $1$ at each point of $\Omega$. Thus, by \cite[VIII.A.Prop.6]{gr2} $H^0(\Omega,\Jhaz(Z')\an_X)=H^0(\Omega,\an_X)$. By \cite[VIII.A.Thm.11]{gr2} there exists a holomorphic function $H\in H^0(X,\Jhaz(Z')\an_X)=\Jhaz(Z')$ that is close to $1$ on $K$, so $\{H=0\}\cap K=\varnothing$ and $H\in\Ss_K$. By \cite[53 A.5]{kk} $\mer(X)\cong\prod_{k\geq1}\mer(X_k)$ and $\mer(Z)\cong\prod_{k=1}^m\mer(X_k)$. Consider the surjective projection homomorphism 
$$
\pi:\prod_{k\geq1}\mer(X_k)\to\prod_{k=1}^m\mer(X_k)
$$
and let $\pi':\mer(X)\to\mer(Z),\ \frac{F}{G}\mapsto\frac{F|_Z}{G|_Z}$ be the corresponding surjective homomorphism (induced by $\pi$ and the isomorphisms above). Let us check that the homomorphism
$$
\theta:\Ss_K^{-1}\mer(X)\to\mer(Z),\ \frac{A/B}{C}\mapsto\frac{A|_Z}{B|_ZC|_Z}
$$
is an isomorphism. As $\pi'$ is surjective, $\theta$ is surjective. Let us check that it is also injective. If $\theta(\frac{A/B}{C})=\frac{A|_Z}{B|_ZC|_Z}=0$, then $A|_Z=0$, so $HA=0$ on $X$ (recall that $H\in\Jhaz(Z')\cap\Ss_K$ was constructed above). Thus, $H\frac{A}{B}=0$, so the quotient $\frac{A/B}{C}=0$ in $\Ss_K^{-1}\mer(X)$ (as $H\in\Ss_K$). Therefore $\theta$ is injective.

By Cartan's Theorem B the restriction homomorphism $\varphi:\an(X)\to\an(Z),\ F\mapsto F|_Z$ is surjective. Consequently, the homomorphism
$$
\varphi':\Ss_K^{-1}\an(X)\to{\Ss_K'}\!^{-1}\an(Z),\ \frac{F}{G}\mapsto\frac{F|_Z}{G|_Z}
$$
is surjective. Let us check that it is also injective. If $\varphi'(\frac{F}{G})=\frac{F|_Z}{G|_Z}=0$, then $F|_Z=0$, so $HF=0$ on $X$ and $\frac{F}{G}=0$. Therefore $\varphi'$ is injective. 

Finally, as $\mer(Z)$ is the total ring of fractions of ${\Ss_K'}\!^{-1}\an(Z)$, we conclude that $\Ss_K^{-1}\mer(X)$ is the total ring of fractions of $\Ss_K^{-1}\an(X)$, as required.
\end{proof}

\begin{lem}\label{denomin}
Assume $K$ is holomorphically convex. Let $F,G\in\an(X)$ be such that $G$ is not a zero divisor and $\frac{F_x}{G_x}\in\an_{X,x}$ for each $x\in K$. Then there exist $F_1,G_1\in\an(X)$ such that $G_1$ is not a zero divisor, $\{G_1=0\}\cap K=\varnothing$ and $\frac{F}{G}=\frac{F_1}{G_1}$.
\end{lem}
\begin{proof}
Consider the sheaf of ideals $\I$ of $\an_X$ whose stalks are $\I_x:=\{H_x\in\an_{X,x}:\ H_x\tfrac{F_x}{G_x}\in\an_{X,x}\}$. By \cite[Lem.3.2]{as} $\I$ is coherent, so there exist by Cartan's Theorem A an open neighborhood $\Omega\subset X$ of $K$ and holomorphic sections $H_1,\ldots,H_r\in H^0(X,\I)$ such that the ideal $\I_x$ is generated by $H_{1,x},\ldots,H_{r,x}$ for each $x\in\Omega$. Let $\{X_k\}_{k\geq1}$ be the collection of the irreducible components of $X$ that do not meet $K$. As $G$ is a non-zero divisor of $\an(X)$, there exists $z_k\in X_k$ such that $G(z_k)\neq0$ for each $k\geq1$. We may assume that $D:=\{z_k\}_{k\geq1}$ is a discrete subset of $X$. Let $Z:=D\cup\supp(\I)$, which is a complex analytic subset of $X$ that does not meet $K$. Shrinking $\Omega$ we may assume that it is a holomorphically convex neighborhood of $K$ in $X$ and that it does not meet $Z$ (see \cite[VII.A.Prop.3]{gr2}). Consider the coherent sheaf $\Fhaz:=\Jhaz(D)\an_X\cap(H_1,\ldots,H_r)\an_X$. We have $\Fhaz_x=\an_{X,x}$ for each $x\in\Omega$, that is, it is generated by $1$ at each point of $\Omega$. Thus, by \cite[VIII.A.Prop.6]{gr2} $H^0(\Omega,\Fhaz)=H^0(\Omega,\an_X)$. By \cite[VIII.A.Thm.11]{gr2} there exists $H\in H^0(X,\Fhaz)\subset H^0(X,\Jhaz(D))\cap H^0(X,\I)$ that is close to $1$ on $K$. Let $m>0$ be such that $G_1:=mH+G$ does not vanishes at any point of $K$. As $H\in H^0(X,\Jhaz(D))$, we have $H|_D=0$, so $G_1$ does not vanish at any point of $D$. Consequently, $G_1$ is not a zero divisor of $\an(X)$ because it is not identically zero at any of the irreducible components of $X$. In addition, $G_1\in H^0(X,\I)$, so $F_1:=G_1\frac{F}{G}\in\an(X)$ and $\frac{F}{G}=\frac{F_1}{G_1}$, as required.
\end{proof}

\begin{proof}[Proof of Theorem \em \ref{excell0}]
The proof is conducted in several steps:

\noindent{\bf Step 1.} \em Assume $X=\C^n$\em. The maximal ideals of $\Ss_K^{-1}\an(\C^n)$ are $\gtn_x:=\gtm_x(\Ss_K^{-1}\an(\C^n))$ where $x\in K$ and $\gtm_x$ is the maximal ideal of $\an(\C^n)$ associated with $x$. In addition, $(\Ss_K^{-1}\an(\C^n))_{\gtn_x}\cong\an(\C^n)_{\gtm_x}$. As the local rings $\an(\C^n)_{\gtm_x}$ are regular for each $x\in K$, also $\Ss_K^{-1}\an(\C^n)$ is regular. Observe that $\Ss_K^{-1}\an(\C^n)/(\gtm_x(\Ss_K^{-1}\an(\C^n)))\cong\C$ for each $x\in K$ and all the maximal ideals of $\Ss_K^{-1}\an(\C^n)$ have the same height $n$. Consider the partial derivatives $\frac{\partial}{\partial x_i}$ and the projections $\pi_i:\C^n\to\C,\ (x_1,\ldots,x_n)\mapsto x_i$. It holds $\frac{\partial}{\partial x_i}\pi_j=\delta_{ij}$ for $1\leq i,j\leq n$. By \cite[40.F, Th.102, p.291]{m} $\Ss_K^{-1}\an(\C^n)$ is an excellent ring.

\noindent{\bf Step 2.} \em Assume $X$ is a complex analytic subset of $\C^n$\em. In this case $\an(X)=\an(\C^n)/\Jhaz(X)$. Denote ${\Ss'}^{-1}_K=\{F\in\an(\C^n):\ \{F=0\}\cap K=\varnothing\}$. As the restriction homomorphism $\an(\C^n)\to\an(X),\ F\mapsto F|_X$ is surjective, it holds
$$
\Ss^{-1}_K\an(X)\cong({\Ss'}^{-1}_K\an(\C^n))/({\Ss'}^{-1}_K\Jhaz(X)),
$$
which is an excellent ring because it is the quotient of an excellent ring by an ideal.

\noindent{\bf Step 3.} \em Assume $(X,\an_X)$ is a Stein space and $K\subset X$ is a holomorphically convex compact set\em. By Lemma \ref{elimsob} we may assume that $K$ meets all the irreducible components of $X$. By \cite[Thm.5, Lem.6]{n4} there exists a proper injective holomorphic map $\varphi:X\to\C^k$ (where $k$ is large enough) such that for each $x\in\Reg(X)\cup K$ there exists an open neighborhood $U$ in $X$ such that $G|_U:U\to G(U)$ is an analytic isomorphism. By Remmert's Theorem $Z:=\varphi(X)$ is a complex analytic subset of $\C^n$ and consequently a Stein space with its canonical analytic structure. 

Denote $K':=\varphi(K)$ and let us check that $K'$ is holomorphically convex. Let $\Omega$ be an Oka-Weil neighborhood of $K$ in $X$ such that $\varphi|_{\Omega}:\Omega\to\varphi(\Omega)$ is biholomorphic \cite[VII.A.Prop.3]{gr2}. As $K$ is holomorphically convex in $X$, by \cite[A.Cor.9]{gr} $K$ is holomorphically convex in $\Omega$. In addition $\varphi(\Omega)$ is an Oka-Weil neighborhood of $K'$ in $Z$ and $K'$ is holomorphically convex in $\varphi(\Omega)$. As $\varphi(\Omega)$ is holomorphically convex, $K'$ is by \cite[VII.A.Cor.9, VIII.A.Thm.11]{gr} holomorphically convex in $Z$. Thus, it is enough to show: \em the rings of fractions $\Ss^{-1}_K\an(X)$ and $\Ss^{-1}_{K'}\an(Z)$ are isomorphic\em, because the second one is by Step 2 an excellent ring. 

By \cite[Lem.3.8]{as} the homomorphism
$$
\varphi^*:\mer(Z)\to\mer(X),\ \tfrac{F}{G}\mapsto\tfrac{F\circ\varphi}{G\circ\varphi}
$$
is an isomorphism. As $K$ meets all the irreducible components of $X$, also $K'$ meets all the irreducible components of $Z$ (recall that $\varphi$ induces a bijection between the irreducible components of $X$ and those of $Z$). Thus, $\Ss^{-1}_K\an(X)\subset\mer(X)$ and $\Ss^{-1}_{K'}\an(Z)\subset\mer(Z)$. We claim: $\varphi^*(\Ss^{-1}_{K'}\an(Z))=\Ss^{-1}_K\an(X)$. 

The inclusion $\varphi^*(\Ss^{-1}_{K'}\an(Z))\subset\Ss^{-1}_K\an(X)$ is clear because $\varphi(K)=K'$. To prove the converse inclusion pick $A\in\an(X)$, $B\in\Ss_K$ and $F,G\in\an(X)$ such that $G$ is a non zero divisor of $\an(X)$ and $\tfrac{F\circ\varphi}{G\circ\varphi}=\frac{A}{B}$. As $B\in\Ss_K$ and $\varphi:X\to Z$ is biholomorphic on an open neighborhood of $K$, we deduce that $\tfrac{F}{G}$ is holomorphic on $K'$. By Lemma \ref{denomin} there exist $F_1,G_1\in\an(Z)$ such that $G_1$ is not a zero divisor of $\an(Z)$, $\{G_1=0\}\cap K'=\varnothing$ and $\frac{F}{G}=\frac{F_1}{G_1}$. As $G_1\in\Ss_{K'}$, we have $\frac{F}{G}=\frac{F_1}{G_1}\in\Ss^{-1}_{K'}\an(Z)$ and the converse inclusion is proved.

\noindent{\bf Step 4.} \em General case: $(X,\an_X)$ is a Stein space and $K\subset X$ is a compact set\em. Let $\hat{K}$ be the holomorphic convex hull of $K$ in $X$. Let $\Tt$ be the homomorphic image of $\Ss_K$ in $\Ss_{\hat{K}}^{-1}\an(X)$. As $\Ss_{\hat{K}}\subset\Ss_K$ and $1\in\Ss_{\hat{K}}$, we have $\Ss_{K}^{-1}\an(X)\cong\Tt^{-1}(\Ss_{\hat{K}}^{-1}\an(X))$, see Remarks \ref{mcs} below. As $\Ss_{\hat{K}}^{-1}\an(X)$ is excellent, the ring of fractions $\Tt^{-1}(\Ss_{\hat{K}}^{-1}\an(X))$ is excellent too, so $\Ss_{K}^{-1}\an(X)$ is excellent, as required. 
\end{proof}

\subsection{Normalization of rings of fractions associated with compact sets.}
Let $(X,\an_X)$ be a Stein space and let $K\subset X$ be a compact set. We compute below the integral closure of $\Ss_K^{-1}\an(X)$ in its total ring of fractions $\Ss_K^{-1}\mer(X)$. We recall some properties of rings of fractions \cite[\S3. Ej. 3, 4, 7, 8, pag. 44]{am} that are used freely in the proofs of Theorems \ref{algebraiccase} and \ref{norexcell} and Lemma \ref{primality0}.
\begin{remarks}\label{mcs}
Let $A$ be a commutative ring. 

(i) If $\Ss,\Tt$ are multiplicatively closed subsets of $A$ and $\Uu$ is the image of $\Tt$ in $\Ss^{-1}A$, then the rings of fractions $(\Ss\Tt)^{-1}A$ and $\Uu^{-1}(\Ss^{-1}A)$ are isomorphic. In particular, if $\Ss\subset\Tt$ and $1\in\Ss$, we have $\Tt^{-1}A\cong\Uu^{-1}(\Ss^{-1}A)$.

(ii) If $\psi:A\to B$ is a homomorphism of rings, $\Ss$ is a multiplicatively closed set and $\Tt:=\psi(\Ss)$, then $\Ss^{-1}B$ and $\Tt^{-1}B$ are isomorphic as $\Ss^{-1}A$-modules.

(iii) A multiplicatively closed set $\Ss$ of $A$ is \em saturated \em if whenever a product $ab$ of elements of $A$ belongs to $\Ss$, then both elements $a,b\in\Ss$. Given a multiplicatively closed set $\Ss\subset A$, there exists a unique smallest saturated multiplicatively closed set $\Ss^*$ containing $\Ss$, which is called the \em saturation of $\Ss$\em. It holds that $\Ss^*$ is the complement in $A$ of the union of the prime ideals of $A$ that do not meet $\Ss$.

(iv) If $\Ss\subset\Tt$ are multiplicatively closed subsets of $A$, then the homomorphism $\phi:\Ss^{-1}A\to\Tt^{-1}A,\ \frac{a}{s}\mapsto\frac{a}{s}$ is an isomorphism if and only if $\Tt\subset\Ss^*$.
\end{remarks}

\begin{thm}\label{norexcell}
Let $(X^\nu,\an_{X^\nu},\pi)$ be the normalization of $(X,\an_X)$ and let $\Tt_{K^*}$ be the multiplicatively closed set of all holomorphic functions on $X^\nu$ that do not vanish at the compact set $K^*:=\pi^{-1}(K)$. Then $\Ss_K^{-1}\ol{\an(X)}^\nu\cong\Tt_{K^*}^{-1}\an(X^\nu)$ is the integral closure of $\Ss_K^{-1}\an(X)$ in its total ring of fractions $\Ss_K^{-1}\mer(X)$ and a finitely generated $\Ss_K^{-1}\an(X)$-module.
\end{thm}
\begin{proof}
First, by the splitting of normalization \cite[1.5.20]{jp} and \cite[33.H, Th.78, p.257]{m} it holds that $\Tt_{K^*}^{-1}\an(X^\nu)$ is a finitely generated $\Ss_K^{-1}\an(X)$-module. The proof of the first part of the statement is conducted in several steps. The first four steps have the purpose of reducing the proof to the case: \em $X$ is irreducible, $K$ is a singleton and $K^*$ is a finite set\em.

\noindent{\bf Step 1.} {\em Reduction to the case in which $X$ has finitely many irreducible components and $K$ meets all of them.} Let $\{X_i\}_{i=1}^m$ be the irreducible components of $X$ that meet $K$. As $\pi:X^\nu\to X$ induces a bijection between the irreducible components of $X$ and $X^\nu$, the irreducible components of $X^\nu$ that meet $K^*$ are $X^\nu_i:=\cl(\pi^{-1}(X_i\setminus\Sing(X)))$ for $i=1,\ldots,m$. Denote ${X^\nu}':=\bigcup_{i=1}^mX^\nu_i$ and $X':=\bigcup_{i=1}^mX_i$. Observe that $(X',\an_{X'}:=\an_X|_{X'})$ and $({X^\nu}',\an_{X^\nu}':=\an_{X^\nu}|_{{X^\nu}'})$ are Stein spaces because they are complex analytic subsets of the Stein spaces $(X,\an_X)$ and $(X^\nu,\an_{X^\nu})$. In addition, $\pi:{X^\nu}'\to X'$ is the normalization of $X'$, so $\mer(X')\cong\mer({X^\nu}')$. By Lemma \ref{elimsob}
\begin{align*}
\Ss_K^{-1}\an(X)\cong{\Ss'}_K^{-1}\an(X')\quad&\text{and}\quad\Tt_{K^*}^{-1}\an(X^\nu)\cong{\Tt'}_{K^*}^{-1}\an({X^\nu}'),\\
\Ss_K^{-1}\mer(X)\cong\mer(X')\quad&\text{and}\quad\Tt_{K^*}^{-1}\mer(X)\cong\mer({X^\nu}'),
\end{align*}
where ${\Ss'}_K:=\{F\in\an(X'):\ \{F=0\}\cap K=\varnothing\}$ and ${\Tt'}_{K^*}:=\{G\in\an({X^\nu}'):\ \{G=0\}\cap K^*=\varnothing\}$ (recall that $\mer(X)\cong\mer(X^\nu)$). Note that ${\Ss'}_K^{-1}\an(X')$ is a subring of $\mer(X')$ and ${\Tt'}_{K^*}^{-1}\an({X^\nu}')$ is a subring of $\mer({X^\nu}')$. Let $\ol{\an(X)}^\nu$ be the integral closure of $\an(X)$ in $\mer(X)$ and let $\ol{\an(X')}^\nu$ be the integral closure of $\an(X')$ in $\mer(X')$. By \cite[Prop. 5.7]{am} the integral closure of $\Ss_K^{-1}\an(X)$ in $\Ss_K^{-1}\mer(X)\cong\mer(X')$ is $\Ss_K^{-1}\ol{\an(X)}^\nu$. Analogously, the integral closure of ${\Ss'}_K^{-1}\an(X')$ in ${\Ss'}_K^{-1}\mer(X')\cong\mer(X')$ is ${\Ss'}_K^{-1}\ol{\an(X')}^\nu$. Consequently $\Ss_K^{-1}\ol{\an(X)}^\nu\cong{\Ss'}_K^{-1}\ol{\an(X')}^\nu$ is the integral closure of $\Ss_K^{-1}\an(X)\cong{\Ss'}_K^{-1}\an(X')$ in $\mer(X')\cong \Ss_K^{-1}\mer(X)$.

Thus, we have to show that ${\Ss'}_K^{-1}\ol{\an(X')}^\nu\cong{\Tt'}_{K^*}^{-1}\an({X^\nu}')$. We assume in the following that $X$ has finitely many irreducible components.

\noindent{\bf Step 2.} {\em Reduction to the case in which $X$ is irreducible.} By the splitting of normalization \cite[1.5.20]{jp} and basic properties of rings of fractions \cite[\S3, \S5]{am} there exist isomorphisms
$$
{\Ss}_K^{-1}\ol{\an(X)}^\nu\cong\ol{{\Ss}_K^{-1}\an(X)}^\nu\cong\prod_{i=1}^m\ol{{\Ss}_K^{-1}(\an(X)/\Jhaz(X_i))}^\nu\cong\prod_{i=1}^m\ol{{\Ss'}_{K_i}^{-1}\an(X_i)}^\nu\cong\prod_{i=1}^m{\Ss'}_{K_i}^{-1}\ol{\an(X_i)}^\nu
$$
where $K_i:=K\cap X_i$ and ${\Ss'}_{K_i}:=\{F\in\an(X_i):\ \{F=0\}\cap K_i=\varnothing\}$. In addition, $\an(X^\nu)\cong\prod_{i=1}^m\an(X^\nu_i)$ (as $X^\nu_1,\ldots,X^\nu_m$ are the connected components of $X^\nu$ because $(X^\nu,\an_{X^\nu})$ is a normal complex analytic space), so
$$
{\Tt}_{K^*}^{-1}\an(X^\nu)\cong{\Tt}_{K^*}^{-1}\Big(\prod_{i=1}^m\an(X^\nu_i)\Big)\cong\prod_{i=1}^m{\Tt}_{K^*}^{-1}\an(X^\nu_i)\cong\prod_{i=1}^m{\Tt}_{K^*_i}^{-1}\an(X^\nu_i)
$$
where $K_i^*:=K^*\cap X^\nu_i=\pi^{-1}(K)\cap X^\nu_i=(\pi|_{X^\nu_i})^{-1}(K\cap X_i)$. Thus, it is enough to show 
$$
{\Ss'}_{K_i}^{-1}\ol{\an(X_i)}^\nu\cong{\Tt}_{K^*_i}^{-1}\an(X^\nu_i),
$$
that is, we assume in the following $(X,\an_X)$ is irreducible. 

\noindent{\bf Step 3.} {\em Reduction to the case in which $K$ and $K^*$ are holomorphically convex.} Let $\hat{K}$ be the holomorphic convex hull of $K$ in $X$. Let $\Ss'$ be the homomorphic image of $\Ss_K$ in $\Ss_{\hat{K}}^{-1}\an(X)$. As $\Ss_{\hat{K}}\subset\Ss_K$ and $1\in\Ss_{\hat{K}}$, we have $\Ss_{K}^{-1}\an(X)\cong\Ss'^{-1}(\Ss_{\hat{K}}^{-1}\an(X))$. 

We claim: \em As $\hat{K}$ is holomorphically convex, $\hat{K}^*:=\pi^{-1}(\hat{K})$ is holomorphically convex too\em.

Pick a point $z\in X^\nu\setminus\hat{K}^*$. Then $\pi(z)\in X\setminus\hat{K}$, so there exists a holomorphic function $F$ on $X$ such that $\sup_{\hat{K}}(F)<|F(\pi(z))|$. Denote $G:=F\circ\pi\in\an(X^\nu)$ and observe that $\sup_{\hat{K}^*}(G)=\sup_{\hat{K}}(F)<|F(\pi(z))|=|G(z)|$. Thus, $z$ does no belong to the holomorphic convex hull of $\hat{K}^*$ in $X^\nu$. Consequently, $\hat{K}^*$ is holomorphically convex.

As $K^*\subset\hat{K}^*$, we have $\Tt_{\hat{K}^*}\subset\Tt_{K^*}$. As $1\in\Tt_{\hat{K}^*}$, we have $\Tt_{K^*}^{-1}\an(X^\nu)\cong\Tt'^{-1}(\Tt_{\hat{K}^*}^{-1}\an(X^\nu))$ where $\Tt'$ is the homomorphic image of $\Tt_{K^*}$ in $\Tt_{\hat{K}^*}^{-1}\an(X^\nu)$. If we prove that $\Tt_{\hat{K}^*}^{-1}\an(X^\nu)$ is the integral closure of $\Ss_{\hat{K}}^{-1}\an(X)$ in $\mer(X)$, we will have by \cite[Prop. 5.12]{am} that $\Ss'^{-1}(\Tt_{\hat{K}^*}^{-1}\an(X^\nu))$ is the integral closure of $\Ss_{K}^{-1}\an(X)\cong\Ss'^{-1}(\Ss_{\hat{K}}^{-1}\an(X))$ in $\mer(X)=\mer(X^\nu)$. As $\Tt'$ is the saturation of $\Ss'$ in $\Tt_{\hat{K}^*}^{-1}\an(X^\nu)$, we deduce
$$
\Tt_{K^*}^{-1}\an(X^\nu)\cong\Tt'^{-1}(\Tt_{\hat{K}^*}^{-1}\an(X^\nu))\cong\Ss'^{-1}(\Tt_{\hat{K}^*}^{-1}\an(X^\nu)).
$$
Consequently, $\Tt_{K^*}^{-1}\an(X^\nu)$ is the integral closure of $\Ss_{K}^{-1}\an(X)$ in $\mer(X)=\mer(X^\nu)$. In the following we assume in addition that $K$ and $K^*$ are holomorphically convex.

\noindent{\bf Step 4.} {\em Reduction to the case in which $K$ is a singleton (and $K^*$ a finite set).} As $K$ and $K^*:=\pi^{-1}(K)$ are holomorphically convex, we have by Corollary \ref{closures}
\begin{align*}
&{\Ss}_{K}^{-1}\ol{\an(X)}^\nu=\bigcap_{x\in K}\ol{\an(X)}^\nu_{\gtm_x},\\ 
&\Tt_{K^*}^{-1}\an(X^\nu)=\bigcap_{y\in K^*}\an(X^\nu)_{\gtn_y}=\bigcap_{x\in K}\bigcap_{y\in\pi^{-1}(x)}\an(X^\nu)_{\gtn_y}=\bigcap_{x\in K}\Tt_{\pi^{-1}(x)}\an(X^\nu).
\end{align*}
By \cite[Prop. 5.12]{am} $\ol{\an(X)}^\nu_{\gtm_x}$ is the integral closure of $\an(X)_{\gtm_x}$ in $\mer(X)$. Thus, it is enough to show: \em $\ol{\an(X)}^\nu_{\gtm_x}$ and $\Tt_{\pi^{-1}(x)}\an(X^\nu)$ are isomorphic under $\pi^*$\em. 

\noindent{\bf Step 5.} Let us prove: {\em $\ol{\an(X)}^\nu_{\gtm_x}\overset{\pi^*}{\cong}\Tt_{\pi^{-1}(x)}\an(X^\nu)$ for each $x\in X$.} 

Write $\pi^{-1}(x)=\{y_1,\ldots,y_r\}$. By Lemma \ref{excell0} $\an(X^\nu)_{\gtn_{y_i}}$ is an excellent ring. We have the following commutative diagram of regular homomorphisms.
$$
\xymatrix{
\an(X^\nu)_{\gtn_{y_i}}\ar@{^{(}->}[r]\ar@{^{(}->}[d]&\an_{X^\nu,y_i}\ar@{^{(}->}[d]\\
\widehat{\an(X^\nu)_{\gtn_{y_i}}}\ar[r]^{\cong}&\widehat{\an_{X^\nu,y_i}}
}
$$
As $X^\nu$ is a normal Stein space, $\an_{X^\nu,y_i}$ is a normal ring. By \cite[VII.2.2(d)]{abr} the completion $\widehat{\an_{X^\nu,y_i}}$ of the local ring $\an_{X^\nu,y_i}$ (with respect to its maximal ideal) is a normal ring. Consequently, $\an(X^\nu)_{\gtn_{y_i}}$ is by \cite[VII.2.2(d)]{abr} a normal ring. The field of fractions of $\an(X^\nu)_{\gtn_{y_i}}$ is $\mer(X^\nu)$ for each $y_i$. By Corollary \ref{closures} and \cite[2.1.15]{hs} the ring of fractions $\Tt_{\pi^{-1}(x)}^{-1}(\an(X^\nu))=\bigcap_{i=1}^r\an(X^\nu)_{\gtn_{y_i}}$ is a normal ring. Denote $\Ss:=\pi^*(\an(X)\setminus\gtm_x)$ and observe that the natural homomorphism $\Ss^{-1}\an(X^\nu)\to\Tt_{\pi^{-1}(x)}^{-1}(\an(X^\nu))$ is an isomorphism because the saturation of $\Ss$ is $\Tt_{\pi^{-1}(x)}$. We have the following commutative diagrams.
$$
\xymatrix{
\an(X)\ar@{^{(}->}[r]^{\pi^*}\ar@{^{(}->}[d]&\an(X^\nu)\ar@{^{(}->}[d]\\
\mer(X)\ar[r]^{\pi^*}_{\cong}&\mer(X^\nu)
}\begin{array}{c}\vspace{-1.5cm}\leadsto\end{array}
\xymatrix{
\an(X)_{\gtm_x}\ar@{^{(}->}[r]^{\pi^*}\ar@{^{(}->}[d]&\Ss^{-1}\an(X^\nu)\ar[r]^(0.4){\cong}&\Tt_{\pi^{-1}(x)}^{-1}(\an(X^\nu))\ar@{^{(}->}[d]\\
\mer(X)\ar[rr]^{\pi^*}_{\cong}&&\mer(X^\nu)
}
$$
As $\Tt_{\pi^{-1}(x)}^{-1}\an(X^\nu)$ is a normal ring, $\ol{\an(X)}^\nu_{\gtm_x}\hookrightarrow\Tt_{\pi^{-1}(x)}^{-1}\an(X^\nu)$. By \cite[Thm.1.1]{abf1} there exists finitely many $H_1,\ldots,H_s\in\an(X^\nu)$ such that 
$$
\Tt_{\pi^{-1}(x)}\an(X^\nu)=H_1\cdot\pi^*(\an(X)_{\gtm_x})+\cdots+H_s\cdot\pi^*(\an(X)_{\gtm_x}).
$$
This means that $\Tt_{\pi^{-1}(x)}\an(X^\nu)$ is a finitely generated $\an(X)_{\gtm_x}$-module, so $\Tt_{\pi^{-1}(x)}\an(X^\nu)$ is by \cite[Prop. 5.1]{am} an integral extension of $\an(X)_{\gtm_x}$, that is, $\Tt_{\pi^{-1}(x)}\an(X^\nu)\hookrightarrow\ol{\an(X)}^\nu_{\gtm_x}$. Consequently, $\Tt_{\pi^{-1}(x)}\an(X^\nu)\cong\ol{\an(X)}^\nu_{\gtm_x}$, as required.
\end{proof}

\subsection{Normalization of an algebraic set endowed with its Stein structure}\label{ssas}
Let $X\subset\C^n$ be an algebraic set and let $\I(X):=\{P\in\C[\x]:\ P|_X=0\}$. We have: \em The equality 
\begin{equation}\label{magic}
\Jhaz_{X,x}:=\{f_x\in\an_{\C^n,x}:\ X_x\subset\{f_x=0\}\}=\I(X)\an_{\C^n,x}
\end{equation}
holds for each $x\in\C^n$\em, see \cite[\S2]{s}. We include here a straightforward proof for the sake of completeness.
\begin{proof}
Pick a point $x\in\C^n$. Assume that it is the origin and let $\gtm$ be the maximal ideal of $\C[\x]$ associated with it. The completion of the local ring $A:=(\C[\x]/\I(X))_{\gtm}$ is $\widehat{A}=\C[[\x]]/(\I(X)\C[[\x]])$. As $A$ is a reduced excellent ring, also $\widehat{A}$ is a reduced excellent ring \cite[VII.2.2(d)]{abr}. Consequently, the ideal $\I(X)\C[[\x]]$ is radical. Consider the local analytic ring $B:=\an_{\C^n,0}/(\I(X)\an_{\C^n,0})$ and observe that its completion is 
$$
\widehat{B}=\C[[\x]]/(\I(X)\an_{\C^n,0}\C[[\x]])=\C[[\x]]/(\I(X)\C[[\x]])=\widehat{A},
$$ 
so $B$ is a reduced excellent ring \cite[VII.2.2(d)]{abr}. Thus, $\I(X)\an_{\C^n,0}$ is a radical ideal. By Hilbert's Nullstellensatz for $\an_{\C^n,0}$ we deduce $\Jhaz_{X,0}=\Jhaz(\ceros(\I(X)\an_{\C^n,0}))=\I(X)\an_{\C^n,x}$, as required.
\end{proof}

As $(\C^n,\an_{\C^n})$ is a Stein manifold and $(X,\an_X:=\an_{\C^n}|_X)$ is a complex analytic subspace, it holds by Cartan's Theorem B that $\an(X):=\an(\C^n)/\Jhaz(X)$ where $\Jhaz(X):=H^0(\C^n,\Jhaz_X)$. In addition $(X,\an_X)$ is a Stein space. Let $P_1,\ldots,P_m$ be a system of generators of $\I(X)$. By \eqref{magic} $P_1,\ldots,P_m$ generate $\Jhaz_{X,x}$ for each $x\in\C^n$. By \cite[VIII.A.Thm.15]{gr2} $P_1,\ldots,P_m$ generate $\Jhaz(X)$ as an $\an(X)$-module, so $\Jhaz(X)=\I(X)\an(X)$. Let $(X^\mu\subset\C^{n+m},\rho)$ be the algebraic normalization of the algebraic set $X\subset\C^n$. 

\begin{proof}[Proof of Theorem \em \ref{algebraiccase}]
As $A:=\C[\x]/\I(X)$ is an excellent ring, the integral closure $\ol{A}^\nu$ of $A$ in its total ring of fractions $Q(A)$ is a finitely generated $A$-module. Let $H_1,\ldots,H_m\in\ol{A}^\nu$ be a finite system of generators of $\ol{A}^\nu$ as an $A$-module. It holds that $\ol{A}^\nu=A[H_1,\ldots,H_m]\cong\C[\x,\y]/\I(X^\mu)$ and $\rho:X^\mu\to X,\ (x,y)\mapsto x$ is the normalization map. 

Fix $x\in X$ and let $R:=A_{\gtm_x}$. Denote $\Ss:=\rho^*(A\setminus\gtm_x)$ and $\Tt:=\ol{A}^\nu\setminus(\gtn_{y_1}\cup\cdots\cup\gtn_{y_r})$ where $\rho^{-1}(x)=\{y_1,\ldots,y_r\}$ and $\gtn_{y_i}$ is the maximal ideal of $\ol{A}^\nu$ associated with $y_i$ for $i=1,\ldots,r$. The integral closure of $R$ in $Q(A)_{\gtm_x}$ is by \cite[Prop. 5.7]{am} $\ol{R}^\nu=\Ss^{-1}\ol{A}^\nu$ and $H_1,\ldots,H_m$ generates $\ol{R}^\nu$ as an $R$-module. The natural homomorphism $\Ss^{-1}\ol{A}^\nu\to\Tt^{-1}\ol{A}^\nu$ is an isomorphism of $R$-modules because the saturation of $\Ss$ in $\ol{A}^\nu$ is $\Tt$. The maximal ideals of the semilocal ring $\ol{R}^\nu=\Tt^{-1}\ol{A}^\nu$ are $\gtn_i:=\gtn_{y_i}\Tt^{-1}\ol{A}^\nu$ for $i=1,\ldots,r$. Let $\gtp_1,\ldots,\gtp_t$ be the minimal prime ideals of the completion $\widehat{R}$. By \cite[VII.3.1]{abr} it holds $t=r$ and (after reordering the indices) $\widehat{\ol{R}^\nu_{\gtn_i}}\cong\ol{(\widehat{R}/\gtp_i)}^\nu$ for $i=1,\ldots,r$ (here is `hidden' the use of Zariski Main Theorem!). As $\ol{R}^\nu\setminus\gtn_i$ is the image of $\ol{A}^\nu\setminus\gtn_{y_i}$ in $\ol{R}^\nu=\Tt^{-1}\ol{A}^\nu$ and $\Tt\subset\ol{A}^\nu\setminus\gtn_{\y_i}$, we have $\ol{R}^\nu_{\gtn_i}\cong\ol{A}^\nu_{\gtn_{y_i}}$ for $i=1,\ldots,r$. Thus, by \cite[1.5.20]{jp}
$$
\ol{\widehat{R}}^\nu\cong\prod_{i=1}^m\ol{(\widehat{R}/\gtp_i)}^\nu\cong\prod_{i=1}^m\widehat{\ol{R}^\nu_{\gtn_i}}\cong\prod_{i=1}^m\widehat{\ol{A}^\nu_{\gtn_{y_i}}}.
$$
By \eqref{magic} $\widehat{\ol{A}^\nu_{\gtn_{y_i}}}\cong\widehat{\an_{X^\mu,y_i}}$ and by \cite[6.1.18]{jp} $\rho_*(\an_{X^\mu})_x\cong\prod_{i=1}^m\an_{X^\mu,y_i}$. Thus, by \cite[24.C, p.174]{m}
$$
\ol{\widehat{R}}^\nu\cong\prod_{i=1}^m\widehat{\an_{X^\mu,y_i}}\cong\widehat{\rho_*(\an_{X^\mu})_x}.
$$
The completion of the ring $\rho_*(\an_{X^\mu})_x$ is considered with respect to its Jacobson radical ideal. As $\rho_*(\an_{X^\mu})_x$ is a finitely generated $\an_{X,x}$-module, we deduce by \cite[23.K, Thm.55, p.170]{m}
$$
\widehat{\rho_*(\an_{X^\mu})_x}\cong\widehat{\an_{X,x}}\otimes_{\an_{X,x}}\rho_*(\an_{X^\mu})_x.
$$
By \eqref{magic} $\widehat{R}\cong\widehat{\an_{X,x}}$. As $\ol{R}^\nu$ is a finitely generated $R$-module, \cite[23.K, Thm.55, p.170]{m} implies
$$
\ol{\widehat{R}}^\nu\cong\ol{R}^\nu\otimes_R\widehat{R}\cong(H_1R+\cdots+H_mR)\otimes_R\widehat{R}\cong H_1\widehat{\an_{X,x}}+\cdots+H_m\widehat{\an_{X,x}}.
$$
As $H_1,\ldots,H_m\in\ol{A}^\nu\subset\an(X^\mu)$, it holds that $H_1\an_{X,x}+\cdots+H_m\an_{X,x}$ is a $\an_{X,x}$-submodule of $\rho_*(\an_{X^\mu})_x$. In addition,
\begin{multline*}
\widehat{\an_{X,x}}\otimes_{\an_{X,x}}\rho_*(\an_{X^\mu})_x\cong\widehat{\rho_*(\an_{X^\mu})_x}\\
\cong\ol{\widehat{R}}^\nu\cong H_1\widehat{\an_{X,x}}+\cdots+H_m\widehat{\an_{X,x}}\cong\widehat{\an_{X,x}}\otimes_{\an_{X,x}}(H_1\an_{X,x}+\cdots+H_m\an_{X,x})
\end{multline*}
As the homomorphism $\an_{X,x}\hookrightarrow\widehat{\an_{X,x}}$ is faithfully flat, $\rho_*(\an_{X^\mu})_x\cong H_1\an_{X,x}+\cdots+H_m\an_{X,x}$. As this happens for each $x\in X$, we deduce by \cite[VIII.A.Thm.15]{gr2} that the polynomials $H_1,\ldots,H_m\in\ol{A}^\nu$ generate $\an(X^\mu)$ as a $\an(X)$-module. As each $H_i$ is integral over $\an(X)$, we conclude $\ol{\an(X)}^\nu\cong\an(X^\mu)$. By Theorems \ref{main} and \ref{normalization0} the tuple $(X^\mu,\an_{X^\mu},\rho)$ is isomorphic to the (analytic) normalization $((X^\nu,\an_{X^\nu}),\pi)$ of $(X,\an_X)$. As $\an(X^\mu)$ is a finitely generated $\an(X)$-module, the same happens with $\an(X^\nu)$, as required.
\end{proof}

\subsection{Integral closure different from normalization ring}
We construct next an irreducible Stein space $(X,\an_X)$ of dimension $2$ such that the ring $\an(X^\nu)$ does not coincide with the integral closure $\ol{\an(X)}^\nu$ of $\an(X)$ in the field $\mer(X)$. 

\begin{example}\label{dim2}
Let $F\in\an(\C)$ be a holomorphic function on $\C$ whose zero-set is $\N^*:=\{2,3,\ldots\}$ and such that ${\rm mult}_k(F)=k$ for each $k\in\N^*$. Let $H\in\an(\C)$ be a holomorphic function on $\C$ whose zero-set is $\N^*$ and such that ${\rm mult}_k(H)=1$ for each $k\in\N^*$.

Consider the open subset $\Omega:=\{{\rm Re}(y-x+\frac{1}{2})>0\}\subset\C^3$, which a Stein manifold endowed with the coherent sheaf $\an_\Omega:=\an_{\C^3}|_{\Omega}$. The pair $(X:=\{F(x)-zF(y)=0\}\cap\Omega,\an_\Omega|_{X})$ is a Stein space. We claim: \em $X$ is irreducible and $\ol{\an(X)}^\nu\neq\an(X^\nu)$.\em
\end{example}
\begin{proof}
As ${\rm mult}_k(F)=k$ for each $k\in\N^*$, we have $\{F=0\}\subset\{F'=0\}$. Consequently, 
$$
\Sing(X)=\{F(x)=0,F(y)=0\}\cap\Omega=\{(k,j,z):\ k,j\in\N^*,\ k\leq j,\ z\in\C\}.
$$ 

Denote $\Theta:=\{{\rm Re}(y-x+\frac{1}{2})>0\}\subset\C^2$. Observe that $\Reg(X):=X\setminus\Sing(X)$ is analytically diffeomorphic to $\Theta\setminus(\C\times\N^*)$ via the biholomorphic map
$$
\Phi:\Reg(X)\to\Theta\setminus(\C\times\N^*),\ (x,y,z)\mapsto(x,y),
$$
whose inverse map is
$$
\Phi^{-1}:\Theta\setminus(\C\times\N^*)\to\Reg(X),\ (x,y)\mapsto\Big(x,y,\frac{F(x)}{F(y)}\Big).
$$
As $\Theta\setminus(\C\times\N^*)$ is connected, $X$ is irreducible. If $a:=(k,j,z_0)\in X$, then the germ $X_a$ is analytically equivalent to the germ at the origin of equation $\{x^k-(z+z_0)y^j=0\}$. We claim: \em if $a:=(k,k,0)$, then $X_a\cong\{x^k-zy^k=0\}$ is an irreducible analytic germ\em. 

If $x^k-zy^k=\lambda(x,y,z)\mu(x,y,z)$ where $\lambda,\mu\in\an_{\C^3,0}$, we may assume $\lambda:=x^r+\lambda_{r+1}$ and $\mu:=x^s+\mu_{s+1}$ where $r,s\geq0$, $r+s=k$, $\lambda_{r+1}\in\gtm^{r+1}$ and $\mu_{s+1}\in\gtm^{s+1}$. Thus,
$$
x^k-zy^k=x^k+x^r\mu_{s+1}+x^s\lambda_{r+1}+\lambda_{r+1}\mu_{s+1}\quad\leadsto\quad -zy^k=x^r\mu_{s+1}+x^s\lambda_{r+1}+\lambda_{r+1}\mu_{s+1}.
$$
As $-zy^k\in\gtm^{k+1}$ and $\lambda_{r+1}\mu_{s+1}\in\gtm^{k+2}$, we deduce $r=0$ or $s=0$, which means that either $\lambda$ or $\mu$ is a unit. Consequently $X_a$ is irreducible.

Let us prove: \em $G(x,y):=\frac{H(x)}{H(y)}\in\mer(X)$ is a weakly holomorphic function on $X$\em. Consequently, \em$\frac{H(x)}{H(y)}$ defines and element of $\an(X^\nu)$\em. 

Fix $a:=(k,j,z_0)\in\Sing(X)$. We have $k\leq j$ and 
$$
\frac{H(x)}{H(y)}=\frac{(x-k)\xi_1(x)}{(y-j)\zeta_1(y)}\quad\text{and}\quad\frac{F(x)}{F(y)}=\frac{(x-k)^k\xi_2(x)}{(y-j)^j\zeta_2(y)}
$$
where $\xi_i(x),\zeta_i(y)$ are units in $\an_{\C^3,a}$. Thus, in $X_a$ it holds that
$$
\frac{H(x)^k}{H(y)^j}\cdot\frac{\zeta_1(y)^j}{\xi_1(x)^k}\cdot\frac{\xi_2(x)}{\zeta_2(y)}=\frac{F(x)}{F(y)}=z.
$$
As $k\leq j$, that is, $j-k\geq0$, we have
$$
\frac{H(x)^k}{H(y)^k}=zH(y)^{j-k}\cdot\frac{\xi_1(x)^k}{\zeta_1(y)^j}\cdot\frac{\zeta_2(y)}{\xi_2(x)}\in\an_{X,a},
$$
so $G(x,y)=\frac{H(x)}{H(y)}$ is locally bounded in a neighborhood of $a$. We conclude that $G(x,y)$ is a weakly holomorphic function on $X$. We claim: \em $G\not\in\ol{\an(X)}^\nu$\em. 

Suppose by contradiction $G\in\ol{\an(X)}^\nu$. There exists $k\geq1$ and $B_0,\ldots,B_{k-1}\in\an(\Omega)$ such that $G^k+\sum_{i=0}^{k-1}B_iG^i$ is identically zero on $X$. For $a:=(k+1,k+1,0)$ the germ $G^k_a+\sum_{i=0}^{k-1}B_{i,a}G^i_a$ is identically zero on $X_a$. Write $H(x)=(x-(k+1))\theta_1(x)$ where $\theta_1(x)$ is a unit in $\an_{\C^3,a}$. The analytic germ
$$
\beta_a:=(x-(k+1))^k+\sum_{i=0}^{k-1}B_{i,a}\cdot(x-(k+1))^i\theta_1(x)^{i-k}(y-(k+1))^{k-i}\theta_1(y)^{k-i}
$$
is identically zero on $X_a$. Write $F(x)=(x-(k+1))^{k+1}\theta_2(x)$ where $\theta_2(x)$ is a unit in $\an_{\C^3,a}$. The ideal of analytic germs vanishing at $X_a$ is generated by
$$
\rho_a:=(x-(k+1))^{k+1}-(y-(k+1))^{k+1}\Big(z\frac{\theta_2(y)}{\theta_2(x)}\Big),
$$
so $\rho_a$ divides $\beta_a$ in $\an_{\C^3,a}$. But this is impossible because the order $k$ of $\beta_a$ at $a$ is smaller than the order $k+1$ of $\rho_a$ at $a$. Consequently $G\not\in\ol{\an(X)}^\nu$, as required.
\end{proof}

\section{Real underlying structure of a complex analytic space}\label{s4}

In this section we develop the main tools we need to prove Theorem \ref{neighnorm0}. Its proof requires some preliminary work concerning the local properties of the real underlying structure of a complex analytic space that have interest by their own. To ease the presentation of some proofs we use both symbols $\ol{\,\cdot\,}$ and $\sigma$ to denote the complex conjugation in $\C^n$. Recall that an ideal $\gta$ of a commutative ring $A$ is called \em real \em if whenever a sum of squares $\sum_{i=1}^pa_i^2$ of elements of $A$ belongs to $\gta$, each $a_i\in\gta$. In particular, real ideals are radical ideals. The \em real radical of the ideal $\gta\subset A$ \em 
$$
\sqrt[r]{\gta}:=\Big\{a\in A:\,a^{2m}+\sum_{j=1}^pa_i^2\in\gta,\ a_i\in A,\ m,p\geq1\Big\}
$$
is the smallest real ideal that contains $\gta$. Of course, an ideal $\gta\subset A$ is real if and only if it coincides with its real radical. A ring $A$ is \em real \em if the zero ideal is a real ideal. In particular, real rings are reduced rings. The \em real reduction \em of a ring $A$ is the quotient $A^{rr}:=A/\sqrt[r]{(0)}$ and it is the greatest real quotient of $A$. If $(X,\an_X)$ is a real analytic space and $(X,\an_X^r)$ is the reduction of $(X,\an_X)$, then $\an_{X,x}^r=(\an_{X,x})^{rr}$ for each $x\in X$. Contrary to what happens in the complex case, the reduction of a real analytic space needs not to be coherent, even if $(X,\an_X)$ is a $C$-analytic space. Consider for instance Whitney's umbrella $X:=\{z^2-x^2y=0\}\subset\R^n$ endowed with its canonical $C$-analytic structure $\an_X:=\an_{\R^n}|_X$.

\subsection{Local algebraic properties of the real underlying structure}
We analyze first the algebraic properties, like the height and the primary decomposition, for certain type of distinguished ideals of $\an_{\R^{2n},x}$ that are constructed from ideals of $\an_{\C^n,x}$. Given an ideal $\gta$ of $\an_{\C^n,x}$, the set $\ol{\gta}:=\{\ol{F_x}:\ F_x\in\gta\}$ is an ideal of $\ol{\an}_{\C^n,x}$. Consider the ideal $\gta^\R:=((\gta\cup\ol{\gta})(\an_{\R^{2n},x}\otimes_\R\C))\cap\an_{\R^{2n},x}$ of $\an_{\R^{2n},x}$. We will prove that the operator $\cdot^\R$ transforms: 
\begin{itemize}
\item a prime ideal $\gtp$ of $\an_{\C^n,x}$ of height $r$ into a real prime ideal $\gtp^\R$ of $\an_{\R^{2n},x}$ of height $2r$ (see Theorem \ref{primality1}), 
\item a radical ideal $\gta$ of $\an_{\C^n,x}$ into an ideal $\gta^\R$ of $\an_{\R^{2n},x}$ such that the primary decomposition of its real radical $\sqrt[r]{\gta^\R}$ can be expressed in terms of the primary decomposition of $\gta$ via the operator $\cdot^\R$ (see Corollary \ref{rradical}).
\end{itemize}

\subsection{Tensor products}
The proofs of the previous results require some preliminary algebraic work that involve tensor products of $\C$-algebras. To make clearer the exposition we use the notation $\an_{\R^{2n},x}\otimes_\R\C$ to refer to $\an_{\C^{2n},x}$ (as complexification of $\an_{\R^{2n},x}$) in order to not make confusion with the initial ring $\an_{\C^n,x}$. The smallest $\C$-subalgebra of $\an_{\R^{2n},x}\otimes_\R\C$ that contains $\an_{\C^n,x}$ and $\ol{\an}_{\C^n,x}$ is
$$
\an_{\C^n,x}\ol{\an}_{\C^n,x}:=\Big\{\sum_{i=1}^pF_{i,x}\ol{G_{i,x}}:\ F_{i,x},G_{i,x}\in\an_{\C^n,x},\ p\geq1\Big\}
$$
endowed with the natural $\C$-algebra structure. Given two ideals $\gta,\gtb$ of $\an_{\C^n,x}$, consider the ideals $\gta\ast\ol{\gtb}:=\gta\ol{\an}_{\C^n,x}+\an_{\C^n,x}\ol{\gtb}$ of $\an_{\C^n,x}\ol{\an}_{\C^n,x}$ and $(\gta\ast\ol{\gtb})^e:=(\gta\ast\ol{\gtb})(\an_{\R^{2n},x}\otimes_\R\C)$ of $\an_{\R^{2n},x}\otimes_\R\C$. Fix two prime ideal $\gtp,\gtq$ of $\an_{\C^n,x}$. 

\begin{lem}\label{cuts}
We have $(\gtp\ast\ol{\gtq})^e\cap\an_{\C^n,x}=\gtp$, $(\gtp\ast\ol{\gtq})^e\cap\ol{\an}_{\C^n,x}=\ol{\gtq}$ and $(\gtp\ast\ol{\gtq})^e\cap(\an_{\C^n,x}\ol{\an}_{\C^n,x})=\gtp\ast\ol{\gtq}$.
\end{lem}

Before proving Lemma \ref{cuts} we need the following preparatory result from Linear Algebra.

\begin{lem}\label{points}
Let $\kappa$ be a field and let $F_1,\ldots,F_r:X\to\kappa$ be $\kappa$-linearly independent functions on a set $X$. If we denote $F:=(F_1,\ldots,F_r)$, there exist points $p_1,\ldots,p_r\in X$ such that the vectors $F(p_1),\ldots,F(p_r)$ are $\kappa$-linearly independent.
\end{lem}
\begin{proof}
For $r=1$ the result is trivially true. Suppose the result true for $r-1$ and let us see that it is also true for $r$. Denote $G:=(F_1,\ldots,F_{r-1})$ and choose by induction points $p_1,\ldots,p_{r-1}\in X$ such that the vectors $G(p_1),\ldots,G(p_{r-1})$ are $\kappa$-linearly independent. Suppose by contradiction that for each $z\in X$ the vectors $F(p_1),\ldots,F(p_{r-1}),F(z)$ are $\kappa$-linearly dependent. Then, for each $z\in X$ there exist scalars $\lambda_1(z),\ldots,\lambda_{r-1}(z)\in\kappa$ such that
$$
F(z)=\lambda_1(z)F(p_1)+\cdots+\lambda_{r-1}(z)F(p_{r-1})\ \leadsto\ G(z)=\lambda_1(z)G(p_1)+\cdots+\lambda_{r-1}(z)G(p_{r-1}).
$$
As the (constant) vectors $G(p_1),\ldots,G(p_{r-1})\in\kappa^{r-1}$ are $\kappa$-linearly independent, each $\lambda_i(z)$ is a $\kappa$-linear combination of the functions $F_1,\ldots,F_{r-1}$. To prove this easily use for instance Cramer's solution for the consistent $\kappa$-linear system
$$
\begin{pmatrix}
F_1(z)\\
\vdots\\
F_{r-1}(z)
\end{pmatrix}=
\begin{pmatrix}
F_1(p_1)&\cdots&F_1(p_{r-1})\\
\vdots&\ddots&\vdots\\
F_{r-1}(p_1)&\cdots&F_{p-1}(p_{r-1})
\end{pmatrix}
\begin{pmatrix}
\lambda_1(z)\\
\vdots\\
\lambda_{r-1}(z)
\end{pmatrix}.
$$
Consequently, $\lambda_j:=\sum_{k=1}^{r-1}\mu_{jk}F_k$ for some $\mu_{jk}\in\kappa$. We conclude
$$
F_r=\sum_{j=1}^{r-1}\Big(\sum_{k=1}^{r-1}\mu_{jk}F_k\Big)F_r(p_j),
$$
which is a contradiction because $F_1,\ldots,F_r$ are $\kappa$-linearly independent. Thus, there exists $p_r\in X$ such that the vectors $F(p_1),\ldots,F(p_r)$ are $\kappa$-linearly independent, as required.
\end{proof}

\begin{proof}[Proof of Lemma \em \ref{cuts}]
The proof is conducted in two steps:

\paragraph{}\label{case1} We prove first: \em $(\gtp\ast\ol{\gtq})^e\cap\an_{\C^n,x}=\gtp$\em. The equality $(\gtp\ast\ol{\gtq})^e\cap\ol{\an}_{\C^n,x}=\ol{\gtq}$ is proved analogously.

Let $F_x\in(\gtp\ast\ol{\gtq})^e\cap\an_{\C^n,x}$ and write $F_x=\sum_{i=1}^r(F_{i,x}G_{i,x}+\ol{A_{i,x}}B_{i,x})$ where $F_{i,x}\in\gtp$, $A_{i,x}\in\gtq$ and $G_{i,x},B_{i,x}\in\an_{\R^{2n},x}\otimes_\R\C$. Define 
$$
G_{i,(x,\ol{x})}'(\z,{\tt w}):=G_{i,x}(\tfrac{\z+{\tt w}}{2},\tfrac{\z-{\tt w}}{2\sqrt{-1}}),B_{i,(x,\ol{x})}'(\z,{\tt w}):=B_{i,x}(\tfrac{\z+{\tt w}}{2},\tfrac{\z-{\tt w}}{2\sqrt{-1}})\in\an_{\C^{2n},(x,\ol{x})}.
$$ 
It holds $G_{i,(x,\ol{x})}'(\z,\ol{\z})=G_{i,x}$ and $B_{i,(x,\ol{x})}'(\z,\ol{\z})=B_{i,x}$.

Let $\Omega\times\sigma(\Omega)\subset\C^{2n}$ be an open connected neighborhood of $(x,\ol{x})$ on which the germs above admit holomorphic representatives $F_i,A_i,G_i',B_i'$. Define $C_i:=\ol{A_i\circ\sigma}\in H^0(\sigma(\Omega),\an_{\C^n})$ and observe $\ol{A_i}=C_i\circ\sigma$. The holomorphic function
$$
\Gamma(z,w)=F(z)-\sum_{i=1}^r(F_i(z)G'_i(z,w)+C_i(w)B_i'(z,w))\in H^0(\Omega\times\sigma(\Omega),\an_{\C^{2n}})
$$
satisfies
\begin{equation*}
\Gamma(z,\ol{z})=F(z)-\sum_{i=1}^r(F_i(z)G_i'(z,\ol{z})+C_i(\ol{z})B_i'(z,\ol{z}))=F(z)-\sum_{i=1}^r(F_i(z)G_i+\ol{A_i(z)})B_i=0.
\end{equation*}
Denote $\gtm_x$ the maximal ideal of $\an_{\C^n,x}$ associated with $x$. By \cite[1.1.5.Prop.1]{d} $\Gamma$ is identically zero on $\Omega\times\sigma(\Omega)$. As $A_i\in\gtq\subset\gtm_x$, we have $C_i(\ol{x})=\ol{A_i(x)}=0$, so
$$
0=\Gamma(z,\ol{x})=F(z)-\sum_{i=1}^r(F_i(z)G_i'(z,\ol{x})+B_i'(z,\ol{x})C_i(\ol{x}))=F(z)-\sum_{i=1}^rF_i(z)G'_i(z,\ol{x}).
$$
Consequently, $F_x=\sum_{i=1}^rF_{i,x}G'_{i,x}(z,\ol{x})\in\gtp$ and $(\gtp\ast\ol{\gtq})^e\cap\an_{\C^n,x}=\gtp$. 

\paragraph{} Next we prove: $(\gtp\ast\ol{\gtq})^e\cap(\an_{\C^n,x}\ol{\an}_{\C^n,x})=\gtp\ast\ol{\gtq}$.

Let $H_x\in(\gtp\ast\ol{\gtq})^e\cap(\an_{\C^n,x}\ol{\an}_{\C^n,x})$ and write as we have done in \ref{case1}
\begin{equation}\label{a0}
H_x=\sum_{i=1}^r(A_{i,x}(\z)B_{i,(x,\ol{x})}(\z,\ol{\z})+C_{i,\ol{x}}(\ol{\z})D_{i,(x,\ol{x})}(\z,\ol{\z}))=\sum_{j=1}^sF_{j,x}(\z)G_{j,\ol{x}}(\ol{\z}),
\end{equation}
where $A_{i,x}\in\gtp$, $\ol{C_{i,\ol{x}}\circ{\sigma}}\in\gtq$, $B_{i,(x,\ol{x})},D_{i,(x,\ol{x})}\in\an_{\C^{2n},(x,\ol{x})}$, $F_{j,x},\ol{G_{j,\ol{x}}\circ\sigma}\in\an_{\C^n,x}$. 
Let $\Omega\times\sigma(\Omega)\subset\C^{2n}$ be an open connected neighborhood of $(x,\ol{x})$ on which the germs above admit holomorphic representatives $A_i,B_i,C_i,D_i,F_j,G_j$. The holomorphic function
\begin{equation}\label{a}
\Gamma(\z,\w):=\sum_{i=1}^rA_i(\z)B_i(\z,{\tt w})+C_i({\tt w})D_i(\z,{\tt w}))-\sum_{j=1}^sF_j(\z)G_j({\tt w}),
\end{equation}
satisfies $\Gamma(\z,\ol{\z})=0$, so by \cite[1.1.5.Prop.1]{d} $\Gamma$ is identically zero.

Consider the $\C$-linear subspace ${\tt H}/\gtq$ of the $\C$-linear space $\an_{\C^n,x}/\gtq$ spanned by $\{\ol{G_{j,\ol{x}}\circ\sigma}:\ j=1,\ldots,s\}$. We may assume that $\{\ol{G_{j,\ol{x}}\circ\sigma}:\ j=1,\ldots,\ell\}$ constitute a basis of ${\tt H}$. Thus, $\ol{G_{k,\ol{x}}\circ\sigma}$ belongs to the $\C$-linear space ${\tt H}+\gtq$ for $k=\ell+1,\ldots,s$. Write 
$$
G_{k,\ol{x}}=\sum_{j=1}^\ell\mu_{jk}G_{j,\ol{x}}+G'_{k,\ol{x}}
$$
where $\mu_{jk}\in\C$ and $\ol{G'_{k,\ol{x}}\circ\sigma}\in\gtq$ for $k=\ell+1,\ldots,s$. We may assume $G'_{k,\ol{x}}$ admits a holomorphic representative $G'_k$ on $\Omega$. Thus,
\begin{multline*}
\sum_{j=1}^sF_j(\z)G_j({\tt w})=\sum_{j=1}^\ell F_j(\z)G_j({\tt w})+\sum_{k=\ell+1}^sF_k(\z)\Big(\sum_{j=1}^\ell\mu_{jk}G_j({\tt w})+G'_k({\tt w})\Big)\\
=\sum_{j=1}^\ell\Big(F_j(\z)+\sum_{k=\ell+1}^s\mu_{jk}F_k(\z)\Big)G_j({\tt w})+\sum_{k=\ell+1}^sF_k(\z)G'_k({\tt w}).
\end{multline*}
If we substitute $F_j$ by $F_j(\z)+\sum_{k=\ell+1}^s\mu_{jk}F_k(\z)$ for $j=1,\ldots,\ell$ and $G_k$ by $G_k'$ for $k=\ell+1,\ldots,s$, we may assume $\ol{G_{k,\ol{x}}\circ\sigma}\in\gtq$ for $k=\ell+1,\ldots,s$. We claim: \em $F_{i,x}\in\gtp$ for $i=1,\ldots,\ell$\em.

After shrinking $\Omega$, we assume that $\ceros(\gtq)$ admits a representative $Y$ that is an irreducible complex analytic subset of $\Omega$. Let $U$ be an open neighborhood in $Y$ of a regular point that is analytically diffeomorphic to an open subset of $\C^d$. As $Y$ is irreducible, the restrictions to $U$ of $\ol{G_1\circ\sigma},\ldots,\ol{G_\ell\circ\sigma}$ are, by the Identity Principle, $\C$-linear independent. Denote $\ol{G\circ\sigma}:=(\ol{G_1\circ\sigma},\ldots,\ol{G_\ell\circ\sigma})$. By Lemma \ref{points} there exist points $p_1,\ldots,p_\ell\in U$ such that the vectors $\ol{G(\ol{p_1})},\ldots,\ol{G(\ol{p_\ell})}$ are $\C$-linearly independent. As $p_k\in U\subset Y$, $\ol{C_{i,\ol{x}}\circ{\sigma}}\in\gtq$ and $\ol{G_{j,\ol{x}}\circ\sigma}\in\gtq$ for $j=\ell+1,\ldots,s$, we deduce after substituting $\ol{p_i}$ in \eqref{a}
$$
0=\Gamma(\z,\ol{p_i})=\sum_{i=1}^rA_i(\z)B_i(\z,\ol{p_i})-\sum_{j=1}^\ell F_j(\z)G_j(\ol{p_i}).
$$
Consequently, 
$$
H_{i,x}:=\sum_{j=1}^\ell F_{j,x}(\z)G_j(\ol{p_i})=\sum_{i=1}^rA_{i,x}(\z)B_{i,x}(\z,\ol{p_i})\in\gtp
$$
for $i=1,\ldots,\ell$. Consider the consistent $\C$-linear system
$$
\begin{pmatrix}
H_{1,x}\\
\vdots\\
H_{\ell,x}
\end{pmatrix}=
\begin{pmatrix}
G_1(\ol{p_1})&\cdots&G_1(\ol{p_\ell})\\
\vdots&\ddots&\vdots\\
G_\ell(\ol{p_1})&\cdots&G_\ell(\ol{p_\ell})
\end{pmatrix}
\begin{pmatrix}
F_{1,x}\\
\vdots\\
F_{\ell,x}
\end{pmatrix}.
$$
As the vectors $G(\ol{p_1}),\ldots,G(\ol{p_\ell})$ are $\C$-linearly independent, $F_{j,x}\in\gtp$ for $j=1,\ldots,\ell$. We conclude 
$$
\sum_{j=1}^sF_{j,x}(\z)G_{j,x}(\ol{\z})=\sum_{j=1}^\ell F_{j,x}(\z)G_{j,x}(\ol{\z})+\sum_{k=\ell+1}^sF_{k,x}(\z)G_{k,x}(\ol{\z})\in\gtp\ol{\an}_{\C^n,x}+{\an}_{\C^n,x}\ol{\gtq}=\gtp\ast\ol{\gtq},
$$
as required.
\end{proof}
\begin{remark}
Consequently, $\an_{\C^n,x}/\gtp$, $\ol{\an}_{\C^n,x}/\ol{\gtq}$ and $(\an_{\C^n,x}\ol{\an}_{\C^n,x})/(\gtp\ast\ol{\gtq})$ can be regarded as $\C$-subalgebras of $(\an_{\R^{2n},x}\otimes_\R\C)/(\gtp\ast\ol{\gtq})^e$. The smallest $\C$-subalgebra of $(\an_{\R^{2n},x}\otimes_\R\C)/(\gtp\ast\ol{\gtq})^e$ that contains $\an_{\C^n,x}/\gtp$ and $\ol{\an}_{\C^n,x}/\ol{\gtq}$ is
$$
(\an_{\C^n,x}/\gtp)(\ol{\an}_{\C^n,x}/\ol{\gtq}):=\Big\{\sum_{i=1}^r[F_{i,x}][\ol{G}_{i,x}]:\ F_{i,x},G_{i,x}\in\an_{\C^n,x}\Big\}\cong(\an_{\C^n,x}\ol{\an}_{\C^n,x})/(\gtp\ast\ol{\gtq}).
$$
\end{remark}

The following result allow us to represent the $\C$-algebra $(\an_{\C^n,x}\ol{\an}_{\C^n,x})/(\gtp\ast\ol{\gtq})$ as the tensor product $(\an_{\C^n,x}/\gtp)\otimes_\C(\ol{\an}_{\C^n,x}/\ol{\gtq})$. The latter description will ease to decide if the $\C$-algebra $(\an_{\C^n,x}\ol{\an}_{\C^n,x})/(\gtp\ast\ol{\gtq})$ is an integral domain, a normal domain, etc.

\begin{lem}\label{tensor}
The map 
$$
\varphi:(\an_{\C^n,x}/\gtp)\otimes_\C(\ol{\an}_{\C^n,x}/\ol{\gtq})\to(\an_{\C^n,x}\ol{\an}_{\C^n,x})/(\gtp\ast\ol{\gtq}),\ \sum_{i=1}^r[F_{i,x}]\otimes[\ol{G}_{i,x}]\mapsto\sum_{i=1}^r[F_{i,x}][\ol{G}_{i,x}]
$$ 
is an isomorphism.
\end{lem}
\begin{proof}
By \cite[Ch.III.\S4.4]{b1} \em $\an_{\C^n,x}/\gtp\otimes_\C\ol{\an}_{\C^n,x}/\ol{\gtq}$ is (isomorphic to) the smallest $\C$-subalgebra of $(\an_{\R^{2n},x}\otimes_\R\C)/(\gtp\ast\ol{\gtq})^e$ that contains $\an_{\C^n,x}/\gtp$ and $\an_{\C^n,x}/\gtq$ if and only if every finite family $\{F_{1,x},\ldots,F_{r,x}\}\subset\an_{\C^n,x}/\gtp$ of $\C$-linearly independent holomorphic germs is also $\ol{\an}_{\C^n,x}/\ol{\gtq}$-linearly independent\em. 

Let $\{F_{1,x}+\gtp,\ldots,F_{r,x}+\gtp\}\subset\an_{\C^n,x}/\gtp\subset(\an_{\R^{2n},x}\otimes_\R\C)/(\gtp\ast\ol{\gtq})^e$ be $\C$-linearly independent. We want to show: \em The set $\{F_{1,x},\ldots,F_{r,x}\}$ is $\ol{\an}_{\C^n,x}/\ol{\gtq}$-linearly independent\em.

The $\C$-linear vector subspace ${\tt H}\subset\an_{\C^n,x}$ generated by $\{F_{1,x},\ldots,F_{r,x}\}$ meets $\gtp$ just in the origin. Let $G_{1,x},\ldots,G_{r,x}\in\ol{\an}_{\C^n,x}$ be such that $\sum_{i=1}^rF_{i,x}\ol{G_{i,x}}\in(\gtp\ast\ol{\gtq})^e$. By Lemma \ref{cuts}
$$
\sum_{i=1}^rF_{i,x}\ol{G_{i,x}}\in(\gtp\ast\ol{\gtq})^e\cap(\an_{\C^n,x}\ol{\an}_{\C^n,x})=\gtp\ast\ol{\gtq}.
$$
Let $A_{j,x}\in\gtp$, $B_{j,x}\in\gtq$ and $C_{j,x},D_{j,x}\in\an_{\C^n,x}$ be such that
$$
\sum_{i=1}^rF_{i,x}\ol{G_{i,x}}=\sum_{j=1}^s(A_{j,x}\ol{C_{j,x}}+\ol{B_{j,x}}D_{j,x}).
$$
Let $G_{i,\ol{x}}',B_{j,\ol{x}}',C_{j,\ol{x}}'\in\an_{\C^n,\ol{x}}$ be holomorphic germs such that $\ol{G_{i,x}}=G_{i,\ol{x}}'(\ol{\z})$, $\ol{B_{j,x}}=B_{j,\ol{x}}'(\ol{\z})$ and $\ol{C_{j,x}}=C_{j,\ol{x}}'(\ol{\z})$. Let $\Omega\subset\C^n$ be an open neighborhood of $x$ such that $F_{i,x},A_{j,x},D_{j,x}$ have representatives in $H^0(\Omega,\an_{\C^n})$ and $G_{i,\ol{x}},C'_{j,\ol{x}},B_{j,\ol{x}}'$ have representatives in $H^0(\sigma(\Omega),\an_{\C^n})$. Shrinking $\Omega$ we may assume in addition that there exists an irreducible complex analytic subset $Y\subset\Omega$ that is a representative of $\ceros(\gtq)$. Consider the holomorphic function
$$
\Gamma(\z,\w):=\sum_{i=1}^rF_i(\z)G_i'({\tt w})-\sum_{j=1}^s(A_j(\z)C'_j({\tt w})+B_j'({\tt w})D_j(\z)).
$$
that satisfies $\Gamma(z,\ol{z})=0$ for each $z\in\Omega$. By \cite[1.1.5.Prop.1]{d} $\Gamma$ is identically zero. We claim: \em $G_{i,x}=\ol{G_{i,\ol{x}}'\circ\sigma}\in\gtq$ for each $i=1,\ldots,r$\em. As $\gtq$ is a prime ideal, Hilbert's Nullstellensatz guarantees $\Jhaz(\ceros(\gtq))=\gtq$. Thus, it is enough to prove: \em $G_i(p)=\ol{G_i'\circ\sigma}(p)=0$ for each $p\in Y$ and each $i=1,\ldots,r$\em.

Pick a point $p\in Y$. As $\ol{B_{j,\ol{x}}'\circ\sigma}=B_{j,x}\in\gtq$,
$$
0=\Gamma(\z,\ol{p})=\sum_{i=1}^rF_i(\z)G_i'(\ol{p})-\sum_{j=1}^sA_j(\z)C'_j(\ol{p}).
$$
We deduce
$$
\sum_{i=1}^rF_i(\z)G_i'(\ol{p})=\sum_{j=1}^sA_j(\z)C'_j(\ol{p})\in{\tt H}\cap\gtp=\{0\}.
$$
As the family $\{F_{1,x},\ldots,F_{r,x}\}$ is $\C$-linearly independent, $G_i'(\ol{p})=0$ for $i=1,\ldots,r$, as required.
\end{proof}

As consequences of Lemma \ref{tensor} we have the following results.

\begin{cor}\label{normal}
Let $\gtp,\gtq$ be prime ideal of $\an_{\C^n,x}$. Then
\begin{itemize}
\item[(i)] $(\an_{\C^n,x}\ol{\an}_{\C^n,x})/(\gtp\ast\ol{\gtq})$ is an integral domain.
\item[(ii)] If in addition the quotients $\an_{\C^n,x}/\gtp$ and $\an_{\C^n,x}/\gtq$ are normal rings, $(\an_{\C^n,x}\ol{\an}_{\C^n,x})/(\gtp\ast\ol{\gtq})$ is a normal integral domain.
\end{itemize}
\end{cor}
\begin{proof}
(i) By \cite[V.\S17, Cor. to Prop.1]{b2} and Lemma \ref{tensor} $(\an_{\C^n,x}\ol{\an}_{\C^n,x})/(\gtp\ast\ol{\gtq})$ is an integral domain. 

(ii) Let $K$ be the field of fractions of $(\an_{\C^n,x}/\gtp)$ and let $\ol{E}$ be the field of fractions of $(\ol{\an}_{\C^n,x}/\ol{\gtq})$. Denote the field of fractions of $(\an_{\C^n,x}\ol{\an}_{\C^n,x})/(\gtp\ast\ol{\gtq})$ with $L$. By \cite[11.6.2]{co} and Lemma \ref{tensor} $K\otimes_\C\ol{E}$ is (isomorphic to) the smallest $\C$-subalgebra $K\ol{E}$ of $L$ that contains $K$ and $\ol{E}$. In addition, $K\otimes_\C(\ol{\an}_{\C^n,x}/\ol{\gtq})$ is by \cite[11.6.2]{co} isomorphic to the smallest $\C$-subalgebra $K(\ol{\an}_{\C^n,x}/\ol{\gtq})$ of $L$ that contains $K$ and $(\ol{\an}_{\C^n,x}/\ol{\gtq})$ whereas $(\an_{\C^n,x}/\gtp)\otimes_\C\ol{E}$ is isomorphic to the smallest $\C$-subalgebra $(\an_{\C^n,x}/\gtp)\ol{E}$ of $L$ that contains $(\an_{\C^n,x}/\gtp)$ and $\ol{E}$. As the homomorphisms $\C\hookrightarrow(\an_{\C^n,x}/\gtp)$, $\C\hookrightarrow K$, $\C\hookrightarrow(\ol{\an}_{\C^n,x}/\ol{\gtq})$ and $\C\hookrightarrow\ol{E}$ are flat,
{\begin{align*}
(\an_{\C^n,x}\ol{\an}_{\C^n,x})/(\gtp\ast\ol{\gtq})\cong(\an_{\C^n,x}/\gtp)\otimes_\C(\ol{\an}_{\C^n,x}/\ol{\gtq})\hookrightarrow K\otimes_\C(\ol{\an}_{\C^n,x}/\ol{\gtq})\hookrightarrow K\otimes_\C\ol{E}\cong K\ol{E}\hookrightarrow L\\
(\an_{\C^n,x}\ol{\an}_{\C^n,x})/(\gtp\ast\ol{\gtq})\cong(\an_{\C^n,x}/\gtp)\otimes_\C(\ol{\an}_{\C^n,x}/\ol{\gtq})\hookrightarrow(\an_{\C^n,x}/\gtp)\otimes_\C\ol{E}\hookrightarrow K\otimes_\C\ol{E}\cong K\ol{E}\hookrightarrow L
\end{align*}}
As $L$ is the field of fractions of $(\an_{\C^n,x}\ol{\an}_{\C^n,x})/(\gtp\ast\ol{\gtq})$, it is also the field of fractions of the integral domains $(\an_{\C^n,x}/\gtp)\ol{E}$, $K(\ol{\an}_{\C^n,x}/\ol{\gtq})$ and $K\ol{E}$. By \cite[II.\S7.7, Cor. to Prop. 14, p. 306]{b1}
\begin{equation}\label{1}
(\an_{\C^n,x}/\gtp)\otimes_\C(\ol{\an}_{\C^n,x}/\ol{\gtq})=((\an_{\C^n,x}/\gtp)\otimes_\C\ol{E})\cap(K\otimes_\C(\ol{\an}_{\C^n,x}/\ol{\gtq}))
\end{equation}
By \cite[6.14.2]{gr} $(\an_{\C^n,x}/\gtp)\otimes_\C\ol{E}$ and $K\otimes_\C(\ol{\an}_{\C^n,x}/\ol{\gtq})$ are normal rings. By \cite[2.1.15]{hs} and equation \eqref{1} the ring $(\an_{\C^n,x}/\gtp)\otimes_\C(\ol{\an}_{\C^n,x}/\ol{\gtq})$ is normal, as required.
\end{proof}

\subsubsection{Analysis of a special case}
Let $(X,\an_X)$ be a reduced complex analytic space. Consider the subsheaf $\an_X\ol{\an}_X$ of the sheaf of rings $\an_X^\R\otimes_{\R}\C$ given by 
$$
\an_{X,x}\ol{\an}_{X,x}:=\{F_{1,x}\ol{G}_{1,x}+\cdots+F_r\ol{G}_{r,x}:\ F_{i,x},G_{i,x}\in\an_{X,x}\}. 
$$
It holds: \em $\an_{X,x}\ol{\an}_{X,x}$ is the smallest $\C$-subalgebra of $\an_{X,x}^\R\otimes_{\R}\C$ that contains both $\an_{X,x}$ and $\ol{\an}_{X,x}$\em.

We rewrite Corollary \ref{normal} as follows.

\begin{cor}\label{normal0}
Let $x\in X$ be such that the ring $\an_{X,x}$ is an integral domain. Then 
\begin{itemize}
\item[(i)] $\an_{X,x}\ol{\an}_{X,x}$ is a integral domain.
\item[(ii)] If in addition $\an_{X,x}$ is an normal ring, $\an_{X,x}\ol{\an}_{X,x}$ is a normal integral domain.
\end{itemize}
\end{cor}

\begin{cor}\label{tensor0}
The map 
$$
\varphi:\an_{X,x}\otimes_\C\ol{\an}_{X,x}\to\an_{X,x}\ol{\an}_{X,x},\ \sum_{i=1}^rF_{i,x}\otimes\ol{G}_{i,x}\mapsto\sum_{i=1}^rF_{i,x}\ol{G}_{i,x}
$$ 
is an isomorphism if and only if the germ $X_x$ is irreducible. Consequently, if $U$ denotes the open set of points $x\in X$ at which the germ $X_x$ is irreducible, the restriction sheaves $(\an_X\ol{\an}_X)|_U$ and $(\an_X\otimes_\C\ol{\an}_X)|_U$ are isomorphic.
\end{cor}
\begin{proof}
If $X_x$ is irreducible the result follows from Lemma \ref{tensor}. Assume next that $X_x$ is reducible. By \cite[Ch.III.\S4.4]{b1} it is enough to find a non identically zero holomorphic germ $F_x\in\an_{X,x}$ that is $\ol{\an}_{X,x}$-linearly dependent. Let $X_{1,x}$ be an irreducible component of $X_x$ and let $F_x\in\an_{X,x}$ be an equation of $X_{1,x}$. Let $G_x$ be an equation of the union of the remaining irreducible components of $X_x$. Observe that $F_x\ol{G}_x=0$, so $F_x$ is $\ol{\an}_{X,x}$-linearly dependent. However, $F_x$ is not identically zero (we use here that $X_x$ is reducible), as required.
\end{proof}

\subsection{Prime ideals}
The clue to prove Theorem \ref{primality1} below, as well as some results concerning the local approach to the underlying structure of the normalization devised in Section \ref{s5}, is the following lemma.

\begin{lem}\label{primality0}
Let $\gtp,\gtq$ be prime ideals of $\an_{\C^n,x}$. Let $\gtP,\gtQ$ be prime ideals of $\an_{\C^m,y}$ such that $\an_{\C^m,y}/\gtP$ is the normalization of $\an_{\C^n,x}/\gtp$ and $\an_{\C^m,y}/\gtQ$ is the normalization of $\an_{\C^n,x}/\gtq$. Then 
\begin{itemize}
\item[(i)] $\gtp\ast\ol{\gtq}$ is a prime ideal of $(\an_{\C^n,x}\ol{\an}_{\C^n,x})$.
\item[(ii)] $(\an_{\C^m,y}\ol{\an}_{\C^m,y})/(\gtP\ast\ol{\gtQ})$ is the normalization of $(\an_{\C^n,x}\ol{\an}_{\C^n,x})/(\gtp\ast\ol{\gtq})$.
\item[(iii)] If $\gtm_x'$ is the maximal ideal of $(\an_{\C^n,x}\ol{\an}_{\C^n,x})/(\gtp\ast\ol{\gtq})$ associated with $x$ and $\gtn_y'$ is the maximal ideal of $(\an_{\C^m,y}\ol{\an}_{\C^m,y})/(\gtP\ast\ol{\gtQ})$ associated with $y$, the local ring $((\an_{\C^m,y}\ol{\an}_{\C^m,y})/(\gtP\ast\ol{\gtQ}))_{\gtn_y'}$ is the normalization of the local ring $((\an_{\C^n,x}\ol{\an}_{\C^n,x})/(\gtp\ast\ol{\gtq}))_{\gtm_x'}$.
\item[(iv)] $(\gtp\ast\ol{\gtq})^e$ is a prime ideal of $\an_{\R^{2n},x}\otimes_\R\C$.
\end{itemize}
\end{lem}
\begin{proof}
By \cite[V.\S17, Cor. to Prop.1]{b2} and Lemma \ref{tensor} $(\an_{\C^n,x}\ol{\an}_{\C^n,x})/(\gtp\ast\ol{\gtq})$ is an integral domain and this proves (i). To prove (iv) we have to show: \em $(\an_{\R^{2n},x}\otimes_\R\C)/(\gtp\ast\ol{\gtq})^e$ is an integral domain\em. 

It is enough to check that the completion of the local ring $(\an_{\R^{2n},x}\otimes_\R\C)/(\gtp\ast\ol{\gtq})^e$ is an integral domain. Let $\gtm_{1,x}$ be the maximal ideal of $\an_{\C^n,x}$. Then $\gtm_x:=\gtm_{1,x}\ol{\an}_{\C^n,x}+\an_{\C^n,x}\ol{\gtm}_{1,x}$ is the maximal ideal of $\an_{\C^n,x}\ol{\an}_{\C^n,x}$ associated with $x$. By \cite[17.9]{na} the completion of $(\an_{\R^{2n},x}\otimes_\R\C)/(\gtp\ast\ol{\gtq})^e$ is (if we assume without loss of generality $x=0$)
$$
\C[[\x,\y]]/((\gtp\ast\ol{\gtq})^e\C[[\x,\y]])
$$ 
and the completion of $(\an_{\C^n,x}\ol{\an}_{\C^n,x})_{\gtm_x}/(\gtp\ast\ol{\gtq})_{\gtm_x}$ is
$$
\C[[\x,\y]]/((\gtp\ast\ol{\gtq})_{\gtm_x}\C[[\x,\y]])=\C[[\x,\y]]/((\gtp\ast\ol{\gtq})\C[[\x,\y]])=\C[[\x,\y]]/((\gtp\ast\ol{\gtq})^e\C[[\x,\y]]).
$$
To show that the completion of $(\an_{\C^n,x}\ol{\an}_{\C^n,x})_{\gtm_x}/(\gtp\ast\ol{\gtq})_{\gtm_x}$ is an integral domain, we prove by \cite[VII.3.1]{abr}: \em the normalization of $(\an_{\C^n,x}\ol{\an}_{\C^n,x})_{\gtm_x}/(\gtp\ast\ol{\gtq})_{\gtm_x}$ is a local ring\em.

Recall that $\an_{\C^m,y}/\gtP$ is the normalization of $\an_{\C^n,x}/\gtp$ and $\an_{\C^m,y}/\gtQ$ is the normalization of $\an_{\C^n,x}/\gtq$. Let $K$ be the field of fractions of $\an_{\C^n,x}/\gtp$ and let $E$ be the field of fractions of $\an_{\C^n,x}/\gtq$. Let $L$ be the field of fractions of $(\an_{\C^n,x}\ol{\an}_{\C^n,x})/(\gtp\ast\ol{\gtq})$.
We have the following commutative diagram:
$$
\xymatrix{
\an_{\C^n,x}/\gtp\ar@{^{(}->}[r]\ar@{^{(}->}[d]&\an_{\C^m,y}/\gtP\ar@{^{(}->}[r]\ar@{^{(}->}[d]&K\ar@{^{(}->}[d]\\
(\an_{\C^n,x}\ol{\an}_{\C^n,x})/(\gtp\ast\ol{\gtq})\ar@{^{(}->}[r]&(\an_{\C^m,y}\ol{\an}_{\C^m,y})/(\gtP\ast\ol{\gtQ})\ar@{^{(}->}[r]&K\ol{E}\ar@{^{(}->}[r]&L\\
\ol{\an}_{\C^n,x}/\ol{\gtq}\ar@{^{(}->}[r]\ar@{^{(}->}[u]&\ol{\an}_{\C^m,y}/\ol{\gtQ}\ar@{^{(}->}[r]\ar@{^{(}->}[u]&\ol{E}\ar@{^{(}->}[u]
}
$$

By Lemma \ref{normal0} the ring $(\an_{\C^m,y}\ol{\an}_{\C^m,y})/(\gtP\ast\ol{\gtQ})$ is normal, so the integral closure 
$$
\ol{(\an_{\C^n,x}\ol{\an}_{\C^n,x})/(\gtp\ast\ol{\gtq})}^\nu\hookrightarrow(\an_{\C^m,y}\ol{\an}_{\C^m,y})/(\gtP\ast\ol{\gtQ}). 
$$
All the elements of $\an_{\C^m,y}/\gtP$ (resp. $\ol{\an}_{\C^m,y}/\ol{\gtQ}$) are integral over $\an_{\C^n,x}/\gtp$ (resp. $\ol{\an}_{\C^n,x}/\ol{\gtq}$), so the elements of $(\an_{\C^m,y}/\gtP)\cup(\ol{\an}_{\C^m,y}/\ol{\gtQ})$ are integral over $(\an_{\C^n,x}\ol{\an}_{\C^n,x})/(\gtp\ast\ol{\gtq})$. By \cite[Cor. 5.3]{am}
$$
\ol{(\an_{\C^n,x}\ol{\an}_{\C^n,x})/(\gtp\ast\ol{\gtq})}^\nu\cong(\an_{\C^m,y}\ol{\an}_{\C^m,y})/(\gtP\ast\ol{\gtQ}),
$$
so we have proved (ii).

Let $\gtm_{1,x}'$ be the maximal ideal of $\an_{\C^n,x}/\gtp$ and $\gtm_{2,x}'$ the maximal ideal of $\an_{\C^n,x}/\gtq$. Then $\gtm_x':=\gtm_{1,x}'(\ol{\an}_{\C^n,x}/\ol{\gtq})+(\an_{\C^n,x}/\gtp)\ol{\gtm}_{2,x}'=\gtm_x/(\gtp\ast\ol{\gtq})$ is the maximal ideal of $(\an_{\C^n,x}\ol{\an}_{\C^n,x})/(\gtp\ast\ol{\gtq})$ associated with $x$. Let $\gtn_{1,y}'$ be the maximal ideal of $\an_{\C^m,y}/\gtP$ and $\gtn_{2,y}'$ the maximal ideal of $\an_{\C^m,y}/\gtQ$. Then $\gtn_y':=\gtn_{1,y}'(\ol{\an}_{\C^m,y}/\ol{\gtQ})+(\an_{\C^m,y}/\gtP)\ol{\gtn}_{2,y}'$ is the maximal ideal of $(\an_{\C^m,y}\ol{\an}_{\C^m,y})/(\gtP\ast\ol{\gtQ})$ associated with $y$. We claim: \em $\gtn_y'$ is the unique prime ideal of $(\an_{\C^m,y}\ol{\an}_{\C^m,y})/(\gtP\ast\ol{\gtQ})$ lying over $\gtm_x'$\em. 

By \cite[Cor. 5.8]{am} $\gtn_{1,y}'$ is the unique prime ideal of $\an_{\C^m,y}/\gtP$ lying over $\gtm_{1,x}'$, because $\an_{\C^m,y}/\gtP$ is a local ring. Let $\gtn'$ be a prime ideal of $(\an_{\C^m,y}\ol{\an}_{\C^m,y})/(\gtP\ast\ol{\gtQ})$ such that $\gtn'\cap(\an_{\C^n,x}\ol{\an}_{\C^n,x})/(\gtp\ast\ol{\gtq})=\gtm_x'$. The prime ideal $\gtn'\cap(\an_{\C^m,y}/\gtP)$ satisfies
\begin{multline*}
(\gtn'\cap(\an_{\C^m,y}/\gtP))\cap(\an_{\C^n,x}/\gtp)=\gtn'\cap(\an_{\C^n,x}/\gtp)\\
=\gtn'\cap(\an_{\C^n,x}\ol{\an}_{\C^n,x})/(\gtp\ast\ol{\gtq})\cap(\an_{\C^n,x}/\gtp)=\gtm_x'\cap(\an_{\C^n,x}/\gtp)=\gtm_{1,x}'.
\end{multline*}
Thus, $\gtn'\cap(\an_{\C^m,y}/\gtP)=\gtn_{1,y}'$. Analogously $\gtn'\cap(\ol{\an}_{\C^m,y}/\ol{\gtQ})=\ol{\gtn}_{2,y}'$, so 
$$
\gtn_y'=\gtn_{1,y}'(\ol{\an}_{\C^m,y}/\ol{\gtQ})+(\an_{\C^m,y}/\gtP)\ol{\gtn}_{2,y}'\subset\gtn'
$$ 
and consequently $\gtn_y'=\gtn'$.

As $(\an_{\C^m,y}\ol{\an}_{\C^m,y})/(\gtP\ast\ol{\gtQ})$ is the integral closure of $(\an_{\C^n,x}\ol{\an}_{\C^n,x})/(\gtp\ast\ol{\gtq})$ in $L$, the integral closure of $A:=(\an_{\C^n,x}\ol{\an}_{\C^n,x})_{\gtm_x}/(\gtp\ast\ol{\gtq})_{\gtm_x}=((\an_{\C^n,x}\ol{\an}_{\C^n,x})/(\gtp\ast\ol{\gtq}))_{\gtm_x'}$ in its field of fractions $L$ is by \cite[Prop. 5.12]{am} $\ol{A}^\nu=((\an_{\C^m,y}\ol{\an}_{\C^m,y})/(\gtP\ast\ol{\gtQ}))_{\gtm_x'}$. The natural homomorphism $\ol{A}^\nu\to((\an_{\C^m,y}\ol{\an}_{\C^m,y})/(\gtP\ast\ol{\gtQ}))_{\gtn_y'}$ is an isomorphism of $A$-modules because the saturation of $((\an_{\C^m,y}\ol{\an}_{\C^m,y})/(\gtP\ast\ol{\gtQ}))\setminus{\gtm_x'}$ is $((\an_{\C^m,y}\ol{\an}_{\C^m,y})/(\gtP\ast\ol{\gtQ}))\setminus{\gtn_y'}$, so we have proved (iii). Consequently, $\ol{A}^\nu$ is a local ring, which proves (iv) after the preparatory reductions. 
\end{proof}

Now we are ready to prove the results announced at the beginning of the section.

\begin{thm}[Height of prime ideals]\label{primality1}
Let $\gtp$ be a prime ideal of $\an_{\C^n,x}$. Then $\gtp^\R$ is a real prime ideal of $\an_{\R^{2n},x}$ and ${\rm ht}(\gtp^\R)=2{\rm ht}(\gtp)$.
\end{thm}
\begin{proof}
Let $r:={\rm ht}(\gtp)$ and let $(0):=\gtp_0\subsetneq\cdots\subsetneq\gtp_r=:\gtp$ be a chain of prime ideals of maximal length in $\an_{\C^n,x}$ (whose ending is $\gtp$). Define for $i=0,\ldots,2r$ the prime ideal
$$
\gtP_i:=\begin{cases}
(\gtp_\ell*\ol{\gtp_\ell})^e&\text{if $i=2\ell$,}\\
(\gtp_{\ell+1}*\ol{\gtp_\ell})^e&\text{if $i=2\ell+1$.}\\
\end{cases}
$$
By Lemma \ref{primality0} $\gtP_0\subsetneq\cdots\subsetneq\gtP_{2r}$ is a chain of prime ideals in $\an_{\R^{2n},x}\otimes_\R\C$ of length $2r$. Thus, $\gtp^\R\otimes_\R\C=(\gtp_\ell*\ol{\gtp_\ell})^e$ has height $\geq 2r$.

It holds: $\dim_\C(\ceros(\gtp))=n-r$ and $\dim_\R(\ceros(\gtp^{\R}))=2(n-r)$. In addition, $\dim_\R(\ceros(\gtp^{\R}))=2n-{\rm ht}(\Jhaz(\ceros(\gtp^{\R})))$ and by \cite[V.\S1.Prop.1 \& Prop.3]{n} 
\begin{multline*}
2n-{\rm ht}(\Jhaz(\ceros(\gtp^{\R}))\otimes_\R\C)=\dim_{\C}(\ceros(\Jhaz(\ceros(\gtp^{\R}))\otimes_\R\C))\\
=\dim_\R(\ceros(\Jhaz(\ceros(\gtp^{\R}))))=\dim_\R(\ceros(\gtp^{\R}))=2n-2r,
\end{multline*}
so ${\rm ht}(\Jhaz(\ceros(\gtp^{\R}))\otimes_\R\C)=2r$. As $\gtp^\R\otimes_\R\C\subset\Jhaz(\ceros(\gtp^{\R}))\otimes_\R\C$, we conclude ${\rm ht}(\gtp^\R\otimes_\R\C)=2r$. 

As $\Jhaz(\ceros(\gtp^{\R}))$ is a real radical ideal, $\Jhaz(\ceros(\gtp^{\R}))\otimes_\R\C$ is a radical ideal. Let $\Jhaz(\ceros(\gtp^{\R}))\otimes_\R\C=\gtq_1\cap\cdots\cap\gtq_s$ be the primary decomposition of $\Jhaz(\ceros(\gtp^{\R}))\otimes_\R\C$ where each $\gtq_j$ is a prime ideal. Assume ${\rm ht}(\gtq_1)={\rm ht}(\Jhaz(\ceros(\gtp^{\R}))\otimes_\R\C)$, so $\gtp^\R\otimes_\R\C\subset\Jhaz(\ceros(\gtp^{\R}))\otimes_\R\C\subset\gtq_1$. As $\gtp^\R\otimes_\R\C$ and $\gtq_1$ are prime ideals of the same height, $\gtp^\R\otimes_\R\C=\gtq_1$. Thus, $\Jhaz(\ceros(\gtp^{\R}))\otimes_\R\C=\gtp^\R\otimes_\R\C$ and, since the homomorphism $\R\hookrightarrow\C$ is faithfully flat, $\Jhaz(\ceros(\gtp^{\R}))=\gtp^\R$, so $\gtp^\R$ is a real prime ideal of height $2r=2n-\dim_\R(\ceros(\gtp^{\R}))$, as required.
\end{proof}

\begin{cor}[Primary decomposition of radical ideals]\label{rradical}
Let $\gta$ be a radical ideal of $\an_{\C^n,x}$ and let $\gta=\gtp_1\cap\cdots\cap\gtp_r$ be the primary decomposition of $\gta$. Then $\sqrt[r]{\gta^\R}=\gtp_1^\R\cap\cdots\cap\gtp_r^\R$ is the primary decomposition of $\sqrt[r]{\gta^\R}$.
\end{cor}
\begin{proof}
As $\ceros(\gta)=\bigcup_{i=1}^r\ceros(\gtp_i)$, we deduce $\ceros(\gta^\R)=\bigcup_{i=1}^r\ceros(\gtp_i^\R)$. By the real Nullstellensatz and Theorem \ref{primality1}
$$
\sqrt[r]{\gta^\R}=\Jhaz(\ceros(\gta^\R))=\bigcap_{i=1}^r\Jhaz(\ceros(\gtp_i^\R))=\bigcap_{i=1}^r\gtp_i^\R.
$$
In addition, $\sqrt[r]{\gta^\R}=\gtp_1^\R\cap\cdots\cap\gtp_r^\R$ is the primary decomposition of $\sqrt[r]{\gta^\R}$.
\end{proof}

\section{Real underlying structure of the normalization}\label{s5}

In this section we prove Theorem \ref{neighnorm0}. Before that we devise some local properties of the real underlying structure of the normalization of a complex analytic space.

\subsection{Local properties of the real underlying structure of the normalization}
Fix a reduced complex analytic space $(X,\an_X)$ and let $(X^\nu,\an_{X^\nu},\pi)$ be its normalization. We analyze in this section the real underlying structure of the normalization $(X^\nu,\an_{X^\nu},\pi)$ of a Stein space $(X,\an_X)$. We will use freely the facts collected in the following lemma. 
\begin{lem}
Let $A$ be an excellent $\R$-algebra and let $\ol{A}^\nu$ be the integral closure of $A$ in its total ring of fraction $Q(A)$. Then $\ol{A}^\nu\otimes_\R\C$ is the integral closure of $A\otimes_\R\C$ in its total ring of fractions $Q(A)\otimes_\R\C$. In addition, if $A\otimes_\R\C$ is a normal ring, then $A$ is a normal ring too. 
\end{lem}
\begin{proof}
The first part of the statement follows from \cite[Prop. 19.1.1 \& Thm. 19.4.3]{hs} or \cite[Prop. 6.14.2]{gr}. For the second part observe first that if $A$ is a $\C$-algebra, then $A\otimes_\R\C\cong A$ and $Q(A)\otimes_\R\C\cong Q(A)$. Otherwise $\sqrt{-1}\not\in A$, $\sqrt{-1}\not\in Q(A)$ and the polynomial $\t^2+1$ is irreducible both in $A[\t]$ and in $Q(A)[\t]$. It holds $A\otimes_\R\C\cong A[\t]/(\t^2+1)$ and $Q(A)\otimes_\R\C\cong Q(A)[\t]/(\t^2+1)$. Using these facts a straightforward exercise shows the remaining part. 
\end{proof}

Our main result of local nature in this section is the following.
\begin{thm}\label{localnormalization}
Let $x\in X$ and write $\pi^{-1}(x):=\{y_1,\ldots,y_r\}$. Let $\gtp_1/\gta,\ldots,\gtp_s/\gta$ be the minimal prime ideals of $\an_{X,x}=\an_{\C^n,x}/\gta$. Then $r=s$ and
\begin{itemize}
\item[(i)] The minimal prime ideals of $(\an_{X,x}^\R)^{rr}:=\an_{\R^{2n},x}/\sqrt[r]{\gta^\R}$ are $\gtp_1^\R/\sqrt[r]{\gta^\R},\ldots,\gtp_r^\R/\sqrt[r]{\gta^\R}$.
\item[(ii)] $\an_{X^\nu,y_i}^\R$ is after reordering the indices the normalization of $\an_{\R^{2n},x}/\gtp_i^\R$ for $i=1,\ldots,r$.
\item[(iii)] The normalization of the reduced ring $(\an_{X,x}^\R)^{rr}$ is $\an_{X^\nu,y_1}^\R\times\cdots\times\an_{X^\nu,y_r}^\R$. 
\end{itemize}
\end{thm}

We approach first the case in which $\an_{X,x}$ is a normal integral domain.

\begin{lem}\label{normal88}
Let $x\in X$ be such that the ring $\an_{X,x}$ is a normal integral domain. Then the ring $\an_{X,x}^\R$ is a normal integral domain and $\widehat{\an_{X,x}^\R\otimes_\R\C}\cong\widehat{(\an_{X,x}\ol{\an}_{X,x})_{\gtm_x'}}$ where $\gtm_x'$ is the maximal ideal of $\an_{X,x}\ol{\an}_{X,x}$ associated with $x$.
\end{lem}
\begin{proof}
Before proving \em $\an_{X,x}^\R$ is a normal ring \em we need some initial preparation. Write $\an_{X,x}\cong\an_{\C^n,x}/\gtp$ where $\gtp$ is a prime ideal of $\an_{\C^n,x}$. It holds by Lemma \ref{primality0} that $\an_{X,x}^\R\cong\an_{\R^{2n},x}/\gtp^\R$ is an integral domain. Let $\gtm_{1,x}$ be the maximal ideal of $\an_{\C_n,x}$ associated with $x$. Then $\gtm_x:=\gtm_{1,x}\ol{\an}_{\C^n,x}+\an_{\C^n,x}\ol{\gtm}_{1,x}$ is the maximal ideal of $\an_{\C^n,x}\ol{\an}_{\C^n,x}$ associated with $x$. Denote $A:=(\an_{\C^n,x}\ol{\an}_{\C^n,x})_{\gtm_x}$ and $B:=(\an_{X,x}\ol{\an}_{X,x})_{\gtm_x'}$ where $\gtm_x'\cong\gtm_x/(\gtp\ast\ol{\gtp})$ is the maximal ideal of $\an_{X,x}\ol{\an}_{X,x}$ associated with $x$. Let us prove: \em $A$ is an excellent ring\em. Once this is proved we deduce by \cite[VII.2.2(b)]{abr} that also $B\cong A/((\gtp\ast\ol{\gtp})A)$ is an excellent ring.

We prove first that $A$ is a local regular ring. By \cite[24.D]{m}, it is enough to show that its completion is regular. The completion of $A$ is $\widehat{A}=\C[[\z,\ol{\z}]]=\C[[\x,\y]]$ where $\z:=(\z_1,\ldots,\z_n)$, $\z_i:=\x_i+\sqrt{-1}\y_i$, $\x:=(\x_1,\ldots,\x_n)$ and $\y:=(\y_1,\ldots,\y_n)$. Consequently, both $A$ and $\widehat{A}$ are local regular rings. 

In addition, the height of the maximal ideal $\gtm_A$ of $A$ is $2n$ by Theorem \ref{primality1}. Observe that $A/\gtm_A\cong\C$ and $\C\hookrightarrow A$. We have in $A$ derivations and elements
$$
D_i:=\begin{cases}
\frac{\partial}{\partial\x_i}&\text{if $i=1,\ldots,n$,}\\
\frac{\partial}{\partial\y_{i-n}}&\text{if $i=n+1,\ldots,2n$}
\end{cases}
\qquad
\xi_i:=\begin{cases}
\x_i&\text{if $i=1,\ldots,n$,}\\
\y_{i-n}&\text{if $i=n+1,\ldots,2n$,}
\end{cases}
$$
such that $D_i\xi_j=\delta_{ij}$ for $i,j=1,\ldots,2n$. By \cite[40.F, Thm. 102, pag. 291]{m} $A$ is an excellent local ring.

We are ready to prove: \em $\an_{X,x}^\R$ is a normal ring\em. By \cite[VII.2.2(d)]{abr}
it is enough to show: \em the completion $\widehat{C}$ of the excellent local ring $C:=\an_{X,x}^\R\otimes_\R\C\cong(\an_{\R^{2n},x}\otimes_\R\C)/(\gtp^\R\otimes_\R\C)$ is normal\em.

Assume without loss of generality $x=0$. The completion of $\an_{\R^{2n},x}\otimes_\R\C$ is $\C[[\x,\y]]$. By \cite[17.9]{na} we have 
\begin{align}
&\widehat{C}=\C[[\x,\y]]/(\gtp^\R\otimes_\R\C)\C[[\x,\y]]=\C[[\x,\y]]/(\gtp\cup\ol{\gtp})\C[[\x,\y]],\label{A1}\\
&\widehat{B}=\widehat{A}/((\gtp\ast\ol{\gtp})\widehat{A})=\C[[\x,\y]]/(\gtp\cup\ol{\gtp})\C[[\x,\y]].\label{A2}
\end{align}
By Lemma \ref{primality0} $B$ is a normal ring. As $B$ is in addition excellent, $\widehat{C}\cong\widehat{B}$ is normal by \cite[VII.2.2(d)]{abr}, as required.
\end{proof}
\begin{remark}
A similar proof shows that if $\an_{X,x}$ is a regular ring, then $\an_{X,x}^\R$ is a regular ring.
\end{remark}

We study now what happens when $\an_{X,x}$ is an integral domain.

\begin{lem}\label{normal76}
Let $x\in X$ be such that $\an_{X,x}$ is an integral domain and write $\pi^{-1}(x)=\{y\}$. Then $\an_{X,x}^\R$ is an integral domain and $\an_{X^\nu,y}^\R$ is the integral closure of $\an_{X,x}^\R$ in its field of fractions.
\end{lem}
\begin{proof}
The proof is conducted in several steps:

\noindent{\bf Step 1.} The ring $\an_{{X^\nu},y}$ is the integral closure of $\an_{X,x}$ in its field of fractions $K$. The $\C$-monomorphism $\pi^*:\an_{X,x}\hookrightarrow\an_{{X^\nu},y}$ induces an $\R$-monomorphism $\an_{X,x}^\R\hookrightarrow\an_{{X^\nu},y}^\R$ (see \cite[II.4.1]{gmt}) and this one a $\C$-monomorphism 
$$
B:=\an_{X,x}^\R\otimes_\R\C\hookrightarrow C:=\an_{{X^\nu},y}^\R\otimes_\R\C.
$$
The fields of fractions of $B$ is contained in the one of $C$. As $C$ is by Lemma \ref{normal88} a normal integral domain, the integral closure $\ol{B}^\nu$ of $B$ in its field of fractions is contained in $C$. We claim: $\ol{B}^\nu=C$. To prove this we assume proved for a while that $C\subset\widehat{\ol{B}^\nu}$.

We have the chain of inclusions $\ol{B}^\nu\hookrightarrow C\hookrightarrow\widehat{\ol{B}^\nu}$. As $B$ is an excellent local ring, the homomorphism $B\to\widehat{B}$ is faithfully flat. If we tensor $\ol{B}^\nu\hookrightarrow C\hookrightarrow\widehat{\ol{B}^\nu}$ by $-\otimes_B\widehat{B}$, injectivity is preserved, that is,
$$
\ol{B}^\nu\otimes_B\widehat{B}\hookrightarrow C\otimes_B\widehat{B}\hookrightarrow\widehat{\ol{B}^\nu}\otimes_B\widehat{B}.
$$
By \cite[23.K, Thm.55, p.170]{m} $\widehat{\ol{B}^\nu}\cong\ol{B}^\nu\otimes_B\widehat{B}$ (because $\ol{B}^\nu$ is a finite $B$-module). Thus,
$$
\widehat{\ol{B}^\nu}\otimes_B\widehat{B}\cong(\ol{B}^\nu\otimes_B\widehat{B})\otimes_B\widehat{B}\cong \ol{B}^\nu\otimes_B\widehat{B}\cong\widehat{\ol{B}^\nu}.
$$
Consequently, 
$$
\widehat{\ol{B}^\nu}\cong \ol{B}^\nu\otimes_B\widehat{B}\hookrightarrow C\otimes_B\widehat{B}\hookrightarrow\widehat{\ol{B}^\nu}\otimes_B\widehat{B}\cong\widehat{\ol{B}^\nu},
$$
so $\ol{B}^\nu\otimes_B\widehat{B}\cong C\otimes_B\widehat{B}$. As the homomorphism $B\to\widehat{B}$ is faithfully flat, we conclude $\ol{B}^\nu=C$.

\noindent{\bf Step 2.} We are reduced to prove: $C\subset\widehat{\ol{B}^\nu}$. As the homomorphism $B\to\widehat{B}$ is regular, we deduce by \cite[VII.2.6]{abr} that $\ol{\widehat{B}}^\nu=\ol{B}^\nu\otimes_B\widehat{B}\cong\widehat{\ol{B}^\nu}$, so we will show: $C\subset\ol{\widehat{B}}^\nu$. As $C\hookrightarrow\widehat{C}$, it is enough to prove: $\widehat{C}\cong\ol{\widehat{B}}^\nu$. 

Define $A:=(\an_{X,x}\ol{\an}_{X,x})_{\gtm_x'}\subset B$ where $\gtm_x'$ is the maximal ideal of $\an_{X,x}\ol{\an}_{X,x}$ associated with $x$. We will prove in two steps: $\widehat{C}\cong\widehat{\ol{A}^\nu}\cong\ol{\widehat{A}}^\nu\cong\ol{\widehat{B}}^\nu$.

\noindent{\bf Step 3.} Let us prove: \em $A$ is a local integral domain with field of fractions $L$ and the integral closure $\ol{A}^\nu\cong(\an_{{X^\nu},y}\ol{\an}_{{X^\nu},y})_{\gtm_x'}$ of $A$ in $L$ is a local integral domain\em.

By Corollary \ref{normal0} $\an_{X,x}\ol{\an}_{X,x}$ is an integral domain whereas $\an_{{X^\nu},y}\ol{\an}_{{X^\nu},y}$ is a normal integral domain. Let $K$ be the total ring of fractions of $\an_{X,x}$ and $\ol{K}$ the field of fractions of $\ol{\an}_{X,x}$. The ring $\an_{{X^\nu},y}$ is the integral closure of $\an_{X,x}$ in $K$ whereas $\ol{\an}_{{X^\nu},y}$ is the integral closure of $\ol{\an}_{X,x}$ in $\ol{K}$. Let $L$ be the field of fractions of $\an_{X^\nu,x}\ol{\an}_{{X^\nu},y}$ and let $K\ol{K}$ be the smallest $\C$-subalgebra $K\ol{K}$ of $L$ that contains $K$ and $\ol{K}$. As $K\ol{K}$ is contained in the field of fractions of $\an_{X,x}\ol{\an}_{X,x}$ and $\an_{X^\nu,x}\ol{\an}_{{X^\nu},y}\subset K\ol{K}\subset L$, we deduce $L$ is the field of fractions of $\an_{X,x}\ol{\an}_{X,x}$. Thus,
$$
\xymatrix{
\an_{X,x}\ol{\an}_{X,x}\ar@{^{(}->}[r]&\an_{{X^\nu},y}\ol{\an}_{{X^\nu},y}\ar@{^{(}->}[r]&K\ol{K}\ar@{^{(}->}[r]&L=\qf(\an_{X,x}\ol{\an}_{X,x}).
}
$$
As the ring $\an_{{X^\nu},y}\ol{\an}_{{X^\nu},y}$ is normal, the integral closure $\ol{\an_{X,x}\ol{\an}_{X,x}}^\nu$ of $\an_{X,x}\ol{\an}_{X,x}$ is contained in $\an_{{X^\nu},y}\ol{\an}_{{X^\nu},y}$. As all the elements of $\an_{{X^\nu},y}$ are integral over $\an_{X,x}$, the elements of $\an_{{X^\nu},y}\cup\ol{\an}_{{X^\nu},y}$ are integral over $\an_{X,x}\ol{\an}_{X,x}$, so $\ol{\an_{X,x}\ol{\an}_{X,x}}^\nu=\an_{{X^\nu},y}\ol{\an}_{{X^\nu},y}$.

We know that $\gtm_x'=\gtm_x\ol{\an}_{X,x}+\an_{X,x}\ol{\gtm}_x$ where $\gtm_x$ is the maximal ideal of $\an_{X,x}$ associated with $x$. Let $\gtn_y'$ be the maximal ideal of $\an_{{X^\nu},y}\ol{\an}_{{X^\nu},y}$ associated to $y$. It holds $\gtn_y'=\gtn_y\ol{\an}_{{X^\nu},y}+\an_{{X^\nu},y}\ol{\gtn}_y$ where $\gtn_y$ is the maximal ideal of $\an_{{X^\nu},y}$ associated with $y$. We claim: \em $\gtn_y'$ is the unique prime ideal $\gtq$ of $\an_{{X^\nu},y}\ol{\an}_{{X^\nu},y}$ such that $\gtq\cap\an_{X,x}\ol{\an}_{X,x}=\gtm_x'$\em. 

Consider the commutative diagram
$$
\xymatrix{
\an_{X,x}\ar@{^{(}->}[r]\ar@{^{(}->}[d]&\an_{{X^\nu},y}\ar@{^{(}->}[r]\ar@{^{(}->}[d]&K\ar@{^{(}->}[d]\\
\an_{X,x}\ol{\an}_{X,x}\ar@{^{(}->}[r]&\an_{{X^\nu},y}\ol{\an}_{{X^\nu},y}\ar@{^{(}->}[r]&K\ol{K}\\
\ol{\an}_{X,x}\ar@{^{(}->}[r]\ar@{^{(}->}[u]&\ol{\an}_{{X^\nu},y}\ar@{^{(}->}[r]\ar@{^{(}->}[u]&\ol{K}\ar@{^{(}->}[u]\\
}
$$
By \cite[5.8]{am} $\gtn_y$ is the unique prime ideal of $\an_{{X^\nu},y}$ lying over $\gtm_x$, because $\an_{{X^\nu},y}$ is a local ring. Let $\gtq$ be a prime ideal of $\an_{{X^\nu},y}\ol{\an}_{{X^\nu},y}$ such that $\gtq\cap\an_{X,x}\ol{\an}_{X,x}=\gtm_x'$. We have
$$
\gtq\cap\an_{X,x}=\gtq\cap\an_{X,x}\ol{\an}_{X,x}\cap\an_{X,x}=\gtm_x'\cap\an_{X,x}=\gtm_x.
$$
The prime ideal $\gtq\cap\an_{{X^\nu},y}$ satisfies $(\gtq\cap\an_{{X^\nu},y})\cap\an_{X,x}=\gtq\cap\an_{X,x}=\gtm_x$, so $\gtq\cap\an_{{X^\nu},y}=\gtn_y$. Analogously $\gtq\cap\ol{\an}_{{X^\nu},y}=\ol{\gtn_y}$, so $\gtn_y'=\gtn_y\ol{\an}_{{X^\nu},y}+\an_{{X^\nu},y}\ol{\gtn}_y\subset\gtq$ and we conclude $\gtn_y'=\gtq$.

The integral closure of $A:=(\an_{X,x}\ol{\an}_{X,x})_{\gtm_x'}$ in $L$ is by \cite[5.12]{am} $\ol{A}^\nu\cong(\an_{{X^\nu},y}\ol{\an}_{{X^\nu},y})_{\gtm_x'}$. We claim: \em $\ol{A}^\nu$ is a local integral domain\em.

Let $\gtn'$ be a maximal ideal of $\ol{A}^\nu$. By \cite[5.8]{am} $\gtn'\cap A=\gtm_x'A$, so $\gtn'=\gtn_y'\ol{A}^\nu$ because $\gtn_y'$ is the unique prime ideal of $\an_{{X^\nu},y}\ol{\an}_{{X^\nu},y}$ lying over $\gtm_x'$. Thus, $\ol{A}^\nu$ is a local integral domain. 

\noindent{\bf Step 4.} We show next: $\widehat{C}\cong\widehat{\ol{A}^\nu}\cong\ol{\widehat{A}}^\nu\cong\ol{\widehat{B}}^\nu$. 

As $\ol{A}^\nu$ is a local ring, $\ol{\widehat{A}}^\nu$ is by \cite[VII.3.1]{abr} a local ring and $\widehat{A}$ is an integral domain. As $A\to\widehat{A}$ is a regular homomorphism and $\ol{A}^\nu$ is a finite $A$-module, $\ol{\widehat{A}}^\nu\cong \ol{A}^\nu\otimes_A\widehat{A}\cong\widehat{\ol{A}^\nu}$ (use \cite[VII.2.6]{abr} and \cite[23.K, Thm.55, p.170]{m}). Using \cite[17.9]{na} one shows that $\widehat{A}$ is the completion of $B$ (proceed as in the last part of the proof of Lemma \ref{normal88}). By Lemma \ref{normal88} $C$ is a normal local domain and $\widehat{C}\cong\widehat{\ol{A}^\nu}$. Consequently, $\widehat{C}\cong\widehat{\ol{A}^\nu}\cong\ol{\widehat{A}}^\nu\cong\ol{\widehat{B}}^\nu$, as required.
\end{proof}

We are ready to prove Theorem \ref{localnormalization}.

\begin{proof}[Proof of Theorem \em \ref{localnormalization}]
We have proved in Lemma \ref{rradical} that $\sqrt[r]{\gta^\R}=\bigcap_{i=1}^s\gtp_i^\R$ is the primary decomposition of the radical ideal $\sqrt[r]{\gta^\R}$. As 
$$
(\an_{X,x}^\R)^{rr}\cong(\an_{X,x}^\R/(\sqrt[r]{\gta^\R}/\gta))\cong\an_{\R^{2n},x}/\sqrt[r]{\gta^\R},
$$ 
we deduce that $\gtp_1^\R/\sqrt[r]{\gta^\R},\ldots,\gtp_s^\R/\sqrt[r]{\gta^\R}$ are the minimal prime ideals of $\an_{X,x}$, which proves (i). By \cite[1.5.20]{jp}
$$
\ol{(\an_{X,x}^\R)^{rr}}^\nu\cong\ol{\an_{\R^{2n},x}/\sqrt[r]{\gta^\R}}^\nu\cong\prod_{i=1}^s\ol{\an_{\R^{2n},x}/\gtp_i^\R}^\nu.
$$
Thus, we have only to check: \em $r=s$ and, after reordering the indices, $\ol{\an_{\R^{2n},x}/\gtp_i^\R}^\nu\cong\an_{X^\nu,y_i}^\R$\em.

By \cite[1.5.20]{jp} we know
$$
\ol{\an_{X,x}}^\nu\cong\ol{\an_{\C^n,x}/\gta}^\nu\cong\prod_{i=1}^s\ol{\an_{\C^n,x}/\gtp_i}^\nu.
$$
By \cite[4.4.8, 6.1.18, 6.3.7]{jp} we deduce $r=s$, $\ol{\an_{X,x}}^\nu\cong\an_{X^\nu,y_1}\times\cdots\times\an_{X^\nu,y_r}$ and after reordering the indices $\ol{\an_{\C^n,x}/\gtp_i}^\nu\cong\an_{X^\nu,y_i}$. By Lemma \ref{normal76} we conclude $\ol{\an_{\R^{2n},x}/\gtp_i^\R}^\nu\cong\an_{X^\nu,y_i}^\R$, which proves (ii) and consequently (iii).
\end{proof}

\subsection{Global properties of the real underlying structure of the normalization}
We are ready to prove Theorem \ref{neighnorm0}.

\begin{proof}[Proof of Theorem \em\ref{neighnorm0}]
We prove the following implications:

(i) $\iff$ (ii) follows from \cite[III.2.15]{gmt} or \cite[Prop.V.\S1.8]{n}.

(i) $\Longrightarrow$ (iii) As $(X^\nu,\an_{X^\nu})$ is a normal complex analytic space, its irreducible components are its connected components and they are pure dimensional, so the same happens with $X^{\nu\,\R}$. By Lemma \ref{normal88} the ring $\an_{X^\nu,x}^\R$ is normal for each $x\in X^{\nu\,\R}$, so the germ $X^{\nu\,\R}_x$ is irreducible for each $x\in X^{\nu\,\R}$. By \cite[V.\S1.Prop.8]{n} the real analytic space $(X^{\nu\,\R},\an_{X^\nu}^\R)$ is coherent, so it is a reduced $C$-analytic space. 

Fix a point $x\in X^\R$. As $(X,\an_X)$ is reduced and $(X^\R,\an_{X}^\R)$ is coherent, $\an_{X,x}^\R=(\an_{X,x}^\R)^{rr}$. By Theorem \ref{localnormalization} the normalization of $\an_{X,x}^\R$ is $\an_{X^\nu,y_1}^\R\times\cdots\times\an_{X^\nu,y_r}^\R$ where $\pi^{-1}(x):=\{y_1,\ldots,y_r\}$. Consequently,
\begin{multline}\label{clue}
\ol{\an_{\widesim{X^\R},x}}^\nu\cong\ol{\an_{X,x}^\R\otimes_\R\C}^\nu\cong\ol{\an_{X,x}^\R}^\nu\otimes_\R\C\\
\cong(\an_{X^\nu,y_1}^\R\otimes_\R\C)\times\cdots\times(\an_{X^\nu,y_r}^\R\otimes_\R\C)\cong\an_{\widesim{X^{\nu\,\R}},y_1}\times\cdots\times\an_{\widesim{X^{\nu\,\R}},y_r}.
\end{multline}
Let $\widesim{\pi^\R}:\widesim{X^{\nu\,\R}}\to\widesim{X^\R}$ be a complexification of $\pi^\R$. By \eqref{clue} we deduce $(\widesim{\pi^\R})^{-1}(X)=X^\nu$. By Lemma \ref{neighcomplex} we may assume that $\widesim{\pi^\R}:\widesim{X^{\nu\,\R}}\to\widesim{X^\R}$ is proper and surjective. By \cite[VI.\S2.Thm. 3]{n} the set of points at which a reduced complex analytic space is normal is an open set, so we may assume that $\widesim{X^{\nu\,\R}}$ is a normal complex analytic space. By \cite[6.1.8]{jp}
$$
\widesim{\pi^\R}_*(\an_{\widesim{X^{\nu\,\R}}})_x\cong\an_{\widesim{X^{\nu\,\R}},y_1}\times\cdots\times\an_{\widesim{X^{\nu\,\R}},y_r}
$$
if $x\in\widesim{X^\R}$ and $(\widesim{\pi^\R})^{-1}(x):=\{y_1,\ldots,y_r\}$. If $x\in X^\R$, we deduce $\widesim{\pi^\R}_*(\an_{\widesim{X^{\nu\,\R}}})_x\cong\ol{\an_{\widesim{X^\R},x}}^\nu$.

Let $(Z,\rho)$ be the normalization of $(\widesim{X^\R},\an_{\widesim{X^\R}})$. By \cite[6.1.8]{jp} $\rho_*(\an_{Z})_x=\ol{\an_{\widesim{X^\R},x}}^\nu$. Thus, the coherent sheaves $\widesim{\pi^\R}_*(\an_{\widesim{X^{\nu\,\R}}})$ and $\rho_*(\an_{Z})$ coincide at the points of $X^\R$. By \cite[\S3.Prop. 2]{c2} we may assume shrinking $\widesim{X^\R}$ if necessary that $\widesim{\pi^\R}_*(\an_{\widesim{X^{\nu\,\R}}})=\rho_*(\an_{Z})$ on $\widesim{X^\R}$. If $x\in\widesim{X^\R}\setminus\Sing(\widesim{X^\R})$, then $\rho^{-1}(x):=\{z\}$ is a singleton and
$$
\an_{Z,z}\cong\rho_*(\an_{Z})_x=\ol{\an_{\widesim{X^\R},x}}^\nu\cong\an_{\widesim{X^\R},x}.
$$
On the other hand,
$$
\an_{Z,z}\cong\rho_*(\an_{Z})_x=\widesim{\pi^\R}_*(\an_{\widesim{X^{\nu\,\R}}})_x\cong\an_{\widesim{X^{\nu\,\R}},y_1}\times\cdots\times\an_{\widesim{X^{\nu\,\R}},y_r}
$$
if $(\widesim{\pi^\R})^{-1}(x):=\{y_1,\ldots,y_r\}$. As $\an_{Z,z}$ is an integral domain, we have $r=1$ and $\widesim{\pi^\R}_*(\an_{\widesim{X^{\nu\,\R}}})_x\cong\an_{\widesim{X^{\nu\,\R}},y_1}\cong\an_{Z,z}\cong\an_{\widesim{X^\R},x}$. Thus, the restriction 
\begin{equation*}
\widesim{\pi^\R}|:\widesim{X^{\nu\,\R}}\setminus(\widesim{\pi^\R})^{-1}(\Sing(\widesim{X^\R}))\to\widesim{X^\R}\setminus\Sing(\widesim{X^\R})
\end{equation*}
is a holomorphic diffeomorphism. Consequently, $(\widesim{X^{\nu\,\R}},\widesim{\pi^\R})$ is isomorphic to the normalization $(Z,\rho)$ of $(\widesim{X^\R},\an_{\widesim{X^\R}})$.

(iii) $\Longrightarrow$ (i) Let $X_i$ be an irreducible component of $X$ and let $X^\nu_i$ be the connected component of $X^\nu$ that satisfies $\pi(X^\nu_i)=X_i$. It holds that $(X^\nu_i,\an_{X^\nu}|_{X^\nu_i},\pi|_{X^\nu_i})$ is the normalization of $(X_i,\an_X|_{X_i})$. Thus, we may assume $X$ is irreducible. Consequently, its normalization $X^\nu$ is connected and both $X^{\nu\,\R}$ and $X^\R$ are irreducible \cite[\S5]{fe1}. By \cite[IV.\S1.Cor. 3]{n} both $X^{\nu\,\R}$ and $X^\R$ are pure dimensional. As $(\widesim{X^{\nu\,\R}},\widesim{\pi_\R},\widesim{\pi^\R})$ is the normalization of $(\widesim{X^\R},\an_{\widesim{X^\R}})$ (after shrinking $\widesim{X^{\nu\,\R}}$ and $\widesim{X^\R}$ if needed) and the sets $\widesim{X^{\nu\,\R}}$ and $\widesim{X^\R}$ are considered initially as `narrow' as needed around $X^{\nu\,\R}$ and $\ol{X}^\R$, we deduce $\widesim{\pi_\R}^{-1}(X^\R)=X^{\nu\,\R}$. Consequently, $X^\R$ is by \cite[IV.3.13]{gmt} coherent, as required.
\end{proof}

\end{document}